\documentclass[11pt, reqno]{amsart}

\usepackage{amssymb,latexsym,amsmath,amsfonts,amsthm}
\usepackage{mathrsfs}
\usepackage{enumitem}
\usepackage[usenames]{color}
\usepackage{hyperref}
\usepackage{comment}
\usepackage{mathtools}
\allowdisplaybreaks

\numberwithin{equation}{section}

\definecolor{DPurple}{rgb}{0.46,0.2,0.69}

\theoremstyle{definition}
\newtheorem{definition}{Definition}[section]

\theoremstyle{remark}
\newtheorem{remark}[definition]{Remark}

\theoremstyle{plain}
\newtheorem{theorem}[definition]{Theorem}
\newtheorem{result}[definition]{Result}
\newtheorem{lemma}[definition]{Lemma}
\newtheorem{proposition}[definition]{Proposition}
\newtheorem{example}[definition]{Example}
\newtheorem{corollary}[definition]{Corollary}

\newcommand{\eps}{\varepsilon}

\newcommand{\zbar}{\overline{z}}
\newcommand{\wbar}{\overline{w}}
\newcommand{\tht}{\theta}

\newcommand{\bas}{\boldsymbol{\epsilon}}
\newcommand{\varz}{{\sf z}}
\newcommand{\vary}{{\sf y}}
\newcommand{\varw}{{\sf w}}
\newcommand{\unitary}{{\sf U}}

\newcommand\mscrf[1]{\widetilde{{\sf F}}_{#1}}

\newcommand{\varf}{{\sf F}}
\newcommand{\varlf}{{{\sf F}^\bullet}}
\newcommand{\tst}{\eta}
\newcommand{\auxu}{{\sf u}}

\newcommand\pd[3]{\frac{\partial^{{#3}}{#1}}{\partial{#2}}}
\newcommand\spd[2]{\partial^{{#2}}_{{#1}}}

\newcommand{\cHess}{\mathfrak{H}_{\raisebox{-2pt}{$\scriptstyle {\mathbb{C}}$}}}

\newcommand{\bdy}{\partial}

\newcommand{\D}{\mathbb{D}}
\newcommand{\Nb}{\mathcal{V}}

\newcommand{\smoo}{\mathcal{C}}
\newcommand{\hol}{\mathcal{O}}
\newcommand\leb[1]{\mathbb{L}^{{#1}}}

\newcommand{\bcdot}{\boldsymbol{\cdot}}
\newcommand{\lrarw}{\longrightarrow}
\newcommand{\rdsc}[1]{\boldsymbol{{\sf r}}_{{#1}}}
\newcommand{\univ}{\boldsymbol{{\sf u}}}
\newcommand{\Leb}{{\sf m}}
\newcommand{\distr}{\mathscr{D}^\prime}

\newcommand{\Z}{\mathbb{Z}}
\newcommand{\N}{\mathbb{N}}

\newcommand{\Cn}{\mathbb{C}^n}

\newcommand{\C}{\mathbb{C}} 
\newcommand{\R}{\mathbb{R}}

\newcommand{\re}{{\sf Re}}
\newcommand{\im}{{\sf Im}}

\newcommand{\wt}{\widetilde}

\begin{document}

\title[Connections: Kobayashi geometry \& potential theory]{On some connections between Kobayashi geometry \\
and pluripotential theory}

\author{Gautam Bharali}
\address{Department of Mathematics, Indian Institute of Science, Bangalore 560012, India}
\email{bharali@iisc.ac.in}

\author{Rumpa Masanta}
\address{Department of Mathematics, Indian Institute of Science, Bangalore 560012, India}
\email{rumpamasanta@iisc.ac.in}

\begin{abstract}
In this paper, we explore some connections between Kobayashi geometry and the Dirichlet problem for the
complex Monge--Amp{\`e}re equation. Among the results we obtain through these connections are: $(i)$~a theorem
on the continuous extension up to $\bdy{D}$ of a proper holomorphic map $F: D\lrarw \Omega$ between domains
of differing dimension, and $(ii)$~a result that establishes the existence of bounded domains with
``nice'' boundary geometry on which H{\"o}lder regularity of the solutions to the complex Monge--Amp{\`e}re
equation fails. The first, a result in Kobayashi geometry, relies upon an auxiliary construction that involves
solving the complex Monge--Amp{\`e}re equation with H{\"o}lder estimates. The second result relies crucially on
a bound for the Kobayashi metric.
\end{abstract}


\maketitle

\vspace{-0.5cm}
\section{Introduction and statement of results}\label{S:intro}
This paper studies certain connections between Kobayashi geometry and the Dirichlet
problem for the complex Monge--Amp{\`e}re equation that are underexplored. Among the results presented in
this paper are the following that are seemingly unrelated:
\begin{itemize}[leftmargin=25pt]
  \item[(1)] A condition on the triple $(D, \Omega, F)$, $F: D\lrarw \Omega$ a proper holomorphic map between
  domains of \textbf{differing} dimension (in which case, nothing usually can be said about the boundary
  regularity of $F$) for $F$ to extend continuously up to $\bdy{D}$.
  \item[(2)] A purely Euclidean condition that\,---\,for $B$-regular domains $\Omega\Subset \Cn$, $n\geq 2$,
  typically having regular boundary\,---\,causes the failure of H{\"o}lder-regularity of the solutions
  to the Dirichlet problem, with smooth data, for the complex Monge--Amp{\`e}re equation on $\Omega$.
\end{itemize}
We refer the reader to Section~\ref{SS:M-A_prelims} for the nuances of the complex Monge--Amp{\`e}re equation
and for some (standard) terminology associated with it. Here, we just state the fact that is essential to understanding the hypotheses of Theorems~\ref{T:monge-ampere_result}
and~\ref{T:complex-geodesic}. Namely, if $\Omega\varsubsetneq \Cn$ is a $B$-regular domain\,---\,see
Section~\ref{SS:M-A_prelims} for a definition of $B$-regularity\,---\,then the following Dirichlet problem
for the complex Monge--Amp{\`e}re equation:
\begin{equation}\label{E:monge-ampere}
  \left.
  \begin{array}{r l}
  \underbrace{dd^cu\wedge\dots \wedge dd^cu}_\text{$n$ factors} =:
    (dd^c{u})^n &\mkern-9mu{= f\beta_n, \; 
    \text{ $u\in \smoo(\overline{\Omega})\cap {\sf psh}(\Omega)$},} \\
    u|_{\bdy\Omega} &\mkern-9mu{= \varphi,}
    \end{array} \right\}
\end{equation}
has a unique solution for any non-negative $f\in \smoo(\overline\Omega; \R)$ and any $\varphi\in
\smoo(\bdy\Omega; \R)$; see \cite[Theorem~4.1]{blocki:cmaohd96}. In \eqref{E:monge-ampere}, $\beta_n$ is
defined as
\[
  \beta_n := (i/2)^n(dz_1\wedge d\zbar_1)\wedge\dots \wedge (dz_n\wedge d\zbar_n).
\]

\subsection{The failure of H{\"o}lder regularity of solutions to the complex Monge--Amp{\`e}re equation}
Our first result is motivated by the observation that very little is known one way or the other,
for $B$-regular domains with $\smoo^2$-smooth boundaries, about H{\"o}lder regularity of the solutions to the
Dirichlet problem \eqref{E:monge-ampere} (even for very nice data) for domains that are \textbf{not}
strongly pseudoconvex. Our result has the implication that a bounded domain $\Omega\varsubsetneq \Cn$, $n\geq
2$, needs to satisfy a small number of geometric conditions, which are rather easy to satisfy simultaneously,
to imply that H{\"o}lder regularity cannot hold for arbitrary data as prescribed by \eqref{E:monge-ampere},
even for $\varphi: \bdy\Omega\lrarw \R$ that is highly regular. Before we state this theorem, we introduce a
notation needed for its statement. With
$\Omega$ as above, we define (here, $\D$ is the open unit disc in $\C$ with centre $0$)
\begin{equation}\label{E:geom}
  \rdsc{\Omega}(z; v) := \sup\big\{r>0: \big(z+ (r\D)v\big) \subset \Omega\big\}   
\end{equation}
for each $z\in \Omega$ and for each $v\in \C^n$ with $\|v\|=1$.
\smallskip

The focus on $B$-regularity in the previous paragraph, and in the hypothesis of the theorem below, is due to
the fact that we are assured of a unique solution to \eqref{E:monge-ampere} when $\Omega$ is $B$-regular.
With those words, we state our first result.

\begin{theorem}\label{T:monge-ampere_result}
Let $\Omega\varsubsetneq \C^n$, $n\geq 2$, be a $B$-regular domain. 
Suppose there exist a sequence $(z_{\nu})_{\nu\geq 1}\subset \Omega$, a point $\xi\in \bdy\Omega$
such that $z_{\nu}\to \xi$, and unit vectors $\univ_{\nu}\in T^{(1,0)}_{z_{\nu}}\Omega$,
$\nu = 1, 2, 3,\dots$, such that
\begin{equation}\label{E:eta-ratio_infty}
  \lim_{\nu\to \infty}\frac{\rdsc{\Omega}(z_{\nu}; \univ_{\nu}\big)}
                        {\big({\rm dist}_{Euc}(z_{\nu}, \bdy\Omega)\big)^{\alpha}} = \infty
                        \quad \forall \alpha\in (0, 1].
\end{equation}
Then, there exist functions $\varphi: \bdy\Omega\lrarw \R$ that are restrictions of
$\smoo^\infty$-smooth functions defined on neighbourhoods of $\bdy\Omega$ such that, for any non-negative
$f\in \smoo(\overline\Omega; \R)$, the unique solution to the Dirichlet problem \eqref{E:monge-ampere}
does not belong to $\smoo^{0,\alpha}(\overline{\Omega})$ for any $\alpha\in (0,1]$.
\end{theorem}

\begin{remark}\label{rem:monge-ampere_geom}
The proof of Theorem~\ref{T:monge-ampere_result} will show that Kobayashi geometry reveals an obstacle for a
domain to admit solutions to \eqref{E:monge-ampere} of H{\"o}lder class. However, the condition
\eqref{E:eta-ratio_infty} involves simple Euclidean measurements.
\end{remark}

\begin{remark}\label{rem:B-reg_geom}
The discussion referring to geometry prior to \eqref{E:geom} suggests that there must be various geometric
conditions implying $B$-regularity. This is indeed so. For $\Omega$ pseudoconvex with
$\smoo^2$-smooth boundary, \cite[Proposition~2.3]{sibony:cdp87} provides a geometric constraint on the set of
weakly pseudoconvex points of $\bdy\Omega$ for $\Omega$ to be $B$-regular. There exist geometric
criteria for $B$-regularity for various special classes of domains: see, for instance,
Proposition~\ref{P:C-strict_convex} and \cite[Proposition~3.1]{dieudunghung:brcd05}.
\end{remark}

Referring once again to the discussion prior to \eqref{E:geom}, one may ask whether the
conditions in Theorem~\ref{T:monge-ampere_result} can simultaneously be true. In this connection, we refer the
reader to Example~\ref{Eg:monge-ampere-not-holder}.%
\smallskip

Theorem~\ref{T:monge-ampere_result} stems, in part, from a connection with Kobayashi geometry\,---\,this is
one of the connections hinted at above. Its proof relies crucially on an estimate for the Kobayashi metric
given by Sibony; see Result~\ref{r:Kob_low_bd_Sibony}.
\smallskip

\subsection{Extension theorems for proper holomorphic maps}\label{SS:extension}
Our principal theorem on the extension of proper holomorphic maps is motivated by the following result.
\begin{result}[paraphrasing {\cite[Corollary~1.5]{forstneric:espdb86}} by Forstneri{\v c}]\label{R:Frostneric}
For each integer $m\geq 1$ there is a proper holomorphic embedding $F:\mathbb{B}^m\lrarw\mathbb{B}^n$,
$n=m+1+2s$ $($where $s=s(m)$ depends only on $m)$, such that $F$ does not extend continuously to
$\overline{\mathbb{B}^m}$.
\end{result}
There is an extensive literature on the continuous extension up to $\bdy D$, whether local or global, of
proper holomorphic maps $F:D\lrarw\Omega$ when $D,\Omega\varsubsetneq\Cn$; i.e., when
$\dim_\C(D)=\dim_\C(\Omega)$.
In this case, often the geometry (resp., local geometry) of $\bdy D$ and $\bdy \Omega$ suffices to ensure
continuous extension of $F$ up to $\bdy D$ (resp., locally); see, for instance, 
\cite{henkin:apnhespd73, pinchuk:phmspd74,diederichfornaess:phmopdwrab79, forstnericrosay:lkmtbcphm87,
berteloot:rlcephm92, sukhov:obrhm94, banik:lcephmlritb24}. Result~\ref{R:Frostneric} suggests that the
situation is starkly different when $\dim_\C(D)<\dim_\C(\Omega)$; also see \cite{low:ephmspdipb85,
dor:phmbboc90}. It is thus natural to ask: \emph{if $F: \mathbb{B}^m\lrarw\mathbb{B}^n$ is a proper
holomorphic map and $m<n$, then what conditions on $F$ would ensure that $F$ extends continuously up to
$\bdy{\mathbb{B}^m}$?} In this setting, owing to the special geometry of Euclidean balls, the focus of
research has been on seeking further information on, or to classify, such $F$\,---\,i.e., on so-called
``rigidity properties'' of $F$; see, for instance, 
\cite{faran:tlphmbbtlcc86, dangelo:phmbbd88, huang:olppmbbcsdd99, huangJi:mbib01, huangJiyin:otgphmbb14}. Such ``rigidity'' theorems need to assume some \emph{a priori} boundary regularity of $F$; this
suggests that the natural question posed above is also a challenging one.
With this as motivation, we wish to investigate\,---\,and not just for mappings between balls\,---\,what
interior conditions on $F:D\lrarw\Omega$, $F$ proper holomorphic and
$\dim_\C(D)<\dim_\C(\Omega)$, enable one to deduce continuous extension of $F$ up to $\bdy D$.
(If $\dim_\C(D)>\dim_\C(\Omega)$, then a holomorphic map $F: D\lrarw \Omega$ is never proper.) Our
first result addresses this problem: it covers not just the Euclidean unit balls but all pairs $(D,\Omega)$
of bounded strongly pseudoconvex domains with $\dim_\C(D)<\dim_\C(\Omega)$ and presents a rather permissive
interior condition on $F$.

\begin{theorem}\label{T:proper-map_str-pseud}
Let $D\varsubsetneq \C^m$ and $\Omega\varsubsetneq \C^n$ be bounded strongly pseudoconvex domains with $m < n$.
Let $F: D\lrarw \Omega$ be a proper holomorphic map, and assume that there exists some $p>m$ such that
\[
  \pd{F_{\mu}}{z_{j}}{}\,\overline{\pd{F_{\nu}}{z_{k}}{}} \in \leb{p}(D, \Leb_{2m})
\]
$($where $\Leb_{2m}$ denotes the $2m$-dimensional Lebesgue measure$)$ for each $j, k: 1\leq j, k\leq m$
and each $\mu, \nu: 1\leq \mu, \nu\leq n$. Then $F$ extends as a continuous map
$\widetilde{F}: (\overline{D}, \bdy{D})\lrarw (\overline{\Omega}, \bdy{\Omega})$.
\end{theorem}

The above is, to the best of our knowledge, the first continuous-extension result for $F$ where $\dim_{\C}(D) < 
\dim_{\C}(\Omega)$ and is not stated for specific examples of $(D, \Omega)$. But we have a result 
that applies to an even larger class of pairs $(D,\Omega)$, $D\varsubsetneq\C^m$ and $\Omega\varsubsetneq\C^n$,
$m,n\in\Z_+$, with $m<n$, of which Theorem~\ref{T:proper-map_str-pseud} is a special case. We refer the reader
to Section~\ref{S:examples} for the definition of the two classes of domains that feature in the following theorem.
What is more useful to see at this point is that domains of either kind are abundant and that each class contains
all bounded strongly pseudoconvex domains; see \cite[Section~2]{charabati:hrscmae15}. With those words, we
present:   
 
\begin{theorem}\label{T:proper-map}
Let $D\varsubsetneq \C^m$ be a strongly hyperconvex Lipschitz domain such that $\bdy{D}$ is a
Lipschitz manifold, let $\Omega\varsubsetneq \C^n$ be a regular strongly hyperconvex Lipschitz domain
such that $\bdy{\Omega}$ is a Lipschitz manifold, and suppose $m < n$. Let $F: D\lrarw \Omega$ be a
proper holomorphic map, and assume that there exists some $p>m$ such that
\begin{equation}\label{E:F_condn}
  \pd{F_{\mu}}{z_{j}}{}\,\overline{\pd{F_{\nu}}{z_{k}}{}} \in \leb{p}(D, \Leb_{2m}).
\end{equation}
for each $j, k: 1\leq j, k\leq m$
and each $\mu, \nu: 1\leq \mu, \nu\leq n$. Then $F$ extends as a continuous map
$\widetilde{F}: (\overline{D}, \bdy{D})\lrarw (\overline{\Omega}, \bdy{\Omega})$.
\end{theorem}

We assume that readers are familiar with the notion of boundaries of open sets as Lipschitz manifolds, but
we provide a definition in Section~\ref{S:geom}.
\smallskip

It is atypical for a pair of domains $(D, \Omega)$ with $\dim_{\C}(D)\leq
\dim_{\C}(\Omega)$ to admit any proper holomorphic map $F : D\lrarw \Omega$. As $D$ and
$\Omega$ in Theorem~\ref{T:proper-map} and $F: D\lrarw \Omega$ must satisfy several conditions,
the question arises: \emph{does there exist any pair $(D,\Omega)$ satisfying these conditions that admits
a proper holomorphic map $F : D\lrarw \Omega$ that satisfies \eqref{E:F_condn}?} (Such questions are left
unanswered in a lot of the literature on the extension of proper holomorphic maps.) With some care,
one can find many examples of triples $(D,\Omega,F)$, even examples where $D$ and $\Omega$ are not strongly
pseudoconvex, that satisfy all the conditions of Theorem~\ref{T:proper-map}. We present such an
example in Section~\ref{S:examples}.
\smallskip

Our proof of Theorem~\ref{T:proper-map}
starts off on an idea of Diederich--Fornaess in
\cite{diederichfornaess:phmopdwrab79}, which involves controlling the operator norm of $F'(z)$, $F$ as
above, as $z$ approaches $\bdy{D}$. In \cite{diederichfornaess:phmopdwrab79} and in
our proof, this relies, in part, on an estimate for the Kobayashi metric on $\Omega$. Thus,
Theorem~\ref{T:proper-map} lies in the realm of Kobayashi geometry. Unlike in 
\cite{diederichfornaess:phmopdwrab79}, $\bdy\Omega$ can be quite rough, but our conditions on $\Omega$ enable
us to appeal to results on the regularity of the solutions to the Dirichlet problem \eqref{E:monge-ampere}
to estimate the Kobayashi metric on $\Omega$. This new approach is the content of Proposition~\ref{P:monge-ampere}
below. But the chief novelty of Theorem~\ref{T:proper-map} lies in dealing with 
the condition
$\dim_\C(D)<\dim_\C(\Omega)$. In many of the works cited above in which $\dim_\C(D) =\dim_\C(\Omega)$, the
above-mentioned control of $\|F'(\bcdot)\|$ also involves the use of a Hopf-type lemma applied to an
auxiliary plurisubharmonic function on $D$. This function is not definable when 
$\dim_\C(D) < \dim_\C(\Omega)$, and the construction of a suitable substitute thereof requires a second, and
careful, appeal to a complex Monge--Amp{\`e}re equation wherein the datum $f\geq 0$ is far less regular than
in \eqref{E:monge-ampere} (see Step~2 of the content of Section~\ref{S:proper-map_proof}). This is yet
another of the connections alluded to at the top of this section. 
\smallskip

Whether there exist any proper holomorphic maps from $D$ to $\Omega$ is no longer an issue when $D = \D$ and 
$\Omega$ is a bounded convex domain. This observation is
due to Lempert when $\Omega$ is strictly convex and has $\smoo^3$-smooth boundary \cite{lempert:mKrdb81}
and to Royden--Wong in the general case \cite{roydenWong:CKmcd83}. For $\Omega$ as in either of these works,
they establish that given any two distinct points $w_1, w_2\in \Omega$, there exists a complex geodesic
$\psi: \D\lrarw \Omega$ such that $w_1, w_2\in \psi(\D)$. A holomorphic map $\psi: \D\lrarw
\Omega$ is called a \emph{complex geodesic} if it is an isometry for the Kobayashi distances on $\D$ and
$\Omega$. Clearly, any complex geodesic is a proper holomorphic map. Whether a complex
geodesic extends as a continuous map on $\overline{\D}$ is a subtle question. This
question was first considered in \cite{lempert:mKrdb81} and answered in the affirmative when $\Omega$ is
as in \cite{lempert:mKrdb81} and is strongly convex. To go beyond the strongly convex case, one needs the
following

\begin{definition}\label{D:C-str-cvx}
A convex domain $\Omega\varsubsetneq \Cn$ is said to be \emph{$\C$-strictly convex} if, for each $p\in
\bdy\Omega$, there exists a support hyperplane of $\Omega$ containing $p$\,---\,denote it as 
$\mathcal{H}_p$\,---\,such that the $\C$-affine hyperplane
\[
  \widetilde{\mathcal{H}}_p :=p+\big((\mathcal{H}_p-p)\cap i(\mathcal{H}_p-p)\big)
\]
satisfies $\widetilde{\mathcal{H}}_p\cap\overline\Omega=\{p\}$.
\end{definition}

The above definition is relevant because of an example by Bharali of a bounded convex domain with 
$\smoo^\infty$-smooth boundary that admits complex geodesics that do \textbf{not} extend as continuous maps on
$\overline{\D}$ \cite[Example~1.2]{bharali:cgtbrhltl16}. The domain in this example is not $\C$-strictly convex.
\smallskip

For $\C$-strictly convex domains, there have been several recent results establishing the continuous extension of
complex geodesics up to $\bdy\D$; see \cite{bharali:cgtbrhltl16, zimmer:cdtlstag17, maitra:ceKi20}. In all
these works, with $\Omega$ as in Definition~\ref{D:C-str-cvx}, $\Omega$ is either assumed to have boundary that
is strictly more regular than $\smoo^1$ or, when $\bdy\Omega$ is $\smoo^1$-smooth, to satisfy a condition
stronger than $\C$-strict convexity. We follow an approach different from these works to establish a
continuous-extension result where $\Omega$ is \textbf{not} assumed to have $\smoo^1$-smooth boundary.

\begin{theorem}\label{T:complex-geodesic}
Let $\Omega\varsubsetneq\Cn$ be a bounded $\C$-strictly convex domain.
\begin{itemize}[leftmargin=25pt]
  \item[$(a)$] Then, $\Omega$ is $B$-regular.
  \item[$(b)$] Assume that the canonical function for $\Omega$ admits a modulus of continuity $\omega$ such that
  $\sqrt{\omega}$ satisfies a Dini condition. Then, every complex geodesic of $\Omega$ extends as a
  continuous map on $\overline{\D}$.
\end{itemize}
\end{theorem}

To prove the above theorem, we appeal to some of the methods used in proving Theorem~\ref{T:proper-map} (which
reduces our proof to an appeal to a Hardy--Littlewood-type result). This is why we need to establish part~$(a)$
above. Theorem~\ref{T:complex-geodesic}-$(a)$, among other things, ensures that the following definition makes
sense.

\begin{definition}
Let $\Omega$ be a $B$-regular domain. The \emph{canonical function} of $\Omega$ is the unique solution to the
Dirichlet problem \eqref{E:monge-ampere} with $f\equiv 0$ and $\varphi := 
\left.-2\|\bcdot\|^2\right|_{\bdy\Omega}$.
\end{definition}


Since we suggested that a novelty of Theorem~\ref{T:complex-geodesic} is that $\Omega$ therein is not
assumed to have $\smoo^1$-smooth boundary, the question arises: \emph{do there exist convex domains
$\Omega\varsubsetneq \Cn$ such that $\bdy\Omega$ is not $\smoo^1$-smooth and satisfy the hypothesis of
Theorem~\ref{T:complex-geodesic}?} The following corollary of Theorem~\ref{T:complex-geodesic} indirectly
(see its proof in Section~\ref{S:complex-geodesic_proofs}) answers this question in the affirmative.

\begin{corollary}\label{C:strongly-convex-geodesic}
Let $\Omega$ be a non-empty finite intersection of bounded strongly convex domains in $\C^n$. 
Then, every complex geodesic of $\Omega$ extends as a continuous map on $\overline{\D}$.    
\end{corollary}
\smallskip

\section{Two examples}\label{S:examples}
This section is devoted to the examples mentioned in Section~\ref{S:intro}. But first, let us fix some
notation that will be used right below and in later sections.

\subsection{Basic notation} The following is a list of recurring notations (some of which were used without
comment in the previous section). In our list, $\Omega$ denotes a domain in $\Cn$.
\begin{itemize}[leftmargin=25pt]
  \item[$(1)$] For $v\in \Cn$, $\|v\|$ will denote the Euclidean norm of $v$.
  \vspace{0.45mm}
  
  \item[$(2)$] Given a non-empty set $S\subseteq \Cn$, the class of all \textbf{real}-valued
  functions that satisfy a H{\"o}lder condition with H{\"o}lder exponent $\alpha$, $\alpha\in (0,1]$
  (to clarify: a H{\"o}lder condition with H{\"o}lder exponent\,$=1$ implies a Lipschitz condition)
  will be denoted by $\smoo^{0,\alpha}(S)$.
  \vspace{0.45mm}

  \item[$(3)$] $\smoo(\overline{\Omega})\cap
  {\sf psh}(\Omega)$ will denote the class of real-valued functions $u: \overline{\Omega}\lrarw \R$
  that are continuous and such that $u|_{\Omega}$ is plurisubharmonic.
  \vspace{0.45mm}

  \item[$(4)$] For $\Omega\varsubsetneq \Cn$, given a point $z\in \Omega$,
  $\delta_{\Omega}(z)$ will denote
  \[
    {\rm dist}_{Euc}(z, \bdy\Omega) := \inf\{\|z-\xi\|: \xi\in \bdy\Omega\}
  \]

  \item[$(5)$] $k_{\Omega}$ will denote the Kobayashi (pseudo)metric on $\Omega$.
  \vspace{0.45mm}

  \item[$(6)$] Given a point $x \in \R^d$ and $r>0$, $B^d(x,r)$ will denote the open Euclidean ball in
  $\R^d$ with radius $r$ and centre $x$.
  \vspace{0.45mm}

  \item[$(7)$] Given a point $z \in \Cn$ and $r>0$, $\mathbb{B}^n(z,r)$ will denote the open Euclidean
  ball in $\Cn$ with radius $r$ and centre $z$. For simplicity, we write $\mathbb{B}^n := \mathbb{B}^n(0,1)$
  and $\D := \mathbb{B}^1(0,1)$ (which already appear in Section~\ref{S:intro}).
\end{itemize}
A final clarification: on several occasions, we will work with $\C$-valued
$\leb{\infty}$-functions on some open set $\mathscr{O}\subseteq \Cn$\,---\,i.e., functions that
are essentially bounded with respect to the Lebesgue measure. For simplicity of notation, we will denote this
class by $\leb{\infty}(\mathscr{O})$.
\smallskip

\subsection{Examples} We will first present the example alluded to prior to \eqref{E:geom} and
right after Remark~\ref{rem:B-reg_geom}. To
this end, we will need the following result:

\begin{result}[Graham, {\cite{graham:dtfhmbcd90}}]\label{R:graham}
Let $\Omega$ be a bounded convex domain in $\Cn$. Then:
\[
  \frac{\|v\|}{2\rdsc{\Omega}(z;v/\|v\|)}\leq k_\Omega(z;v)\leq
  \frac{\|v\|}{\rdsc{\Omega}(z;v/\|v\|)}\quad 
  \forall z\in\Omega \; \; \text{and} \; \; \forall v\in\Cn\setminus\{0\}.
\]
\end{result}

The following proposition will be needed in discussing Example~\ref{Eg:monge-ampere-not-holder}.

\begin{proposition}\label{P:examples-convex}
Let $\varphi:(0,+\infty)\lrarw (0,+\infty)$ be a function of class $\smoo^k$, $k\geq 2$, such that
\begin{itemize}[leftmargin=25pt]
  \item[$(a)$] $\varphi^{(j)}$ extends continuously to $0$ for $0\leq j\leq k$.
  \item[$(b)$] $\lim_{x\to 0^+} (\varphi^{(j-1)}(x)-\varphi^{(j-1)}(0))/x=\varphi^{(j)}(0)$ for $1\leq j\leq k$.
  \item[$(c)$] $\varphi^{(j)}(0)=0$ for $0\leq j\leq s$ for some $s: 1\leq s\leq k$.
  \item[$(d)$] $\varphi$ is strictly increasing and convex.
  \item[$(e)$] $\varphi(a)=1$ for some $a>0$ and the behaviour of $\varphi$ near $a$ is such that 
  \[
    \Omega_\varphi :=\Big\{z\in\Cn : \varphi(z_1\zbar_1)+\sum_{2\leq j\leq n} |z_j|^2<1\Big\}
  \]
  is a convex domain with $\smoo^k$-smooth boundary.
\end{itemize}
Then, $\Omega_\varphi$ is $B$-regular. Fix $l=2,\dots,n$. For each $\varepsilon\in (-1,-1/2)$, write
\[
  z_\varepsilon := (\!\!\!\!\underbrace{0,\dots ,0}_\text{$(l-1)$ entries}\!\!\!\!, \varepsilon,0,\dots,0).
\]
Let $\bas_1 := (1, 0,\dots, 0)$.
Then, for every
  $\varepsilon\in (-1,-1/2)$:
\begin{itemize}
  \item[$(i)$] $\big(\varphi^{-1}(\delta_{\Omega_\varphi}(z_\varepsilon))\big)^{1/2}\leq \rdsc{\Omega_\varphi}
  (z_\varepsilon;\bas_1)\leq 2\big(\varphi^{-1}(\delta_{\Omega_\varphi}(z_\varepsilon))\big)^{1/2}$.
  \item[$(ii)$] Furthermore, we have the estimate
  \[
    \frac{\|v\|}{4\big(\varphi^{-1}(\delta_{\Omega_\varphi}(z_\varepsilon))\big)^{1/2}}\leq k_{\Omega_\varphi}
    (z_\varepsilon;v)\leq
    \frac{\|v\|}{\big(\varphi^{-1}(\delta_{\Omega_\varphi}(z_\varepsilon))\big)^{1/2}}\quad\forall v\in{\rm span}_\C\{\bas_1\}. 
  \]
\end{itemize}
\end{proposition}

\begin{remark}\label{rem:proposition-examples-convex}
The fullest strength of condition $(c)$ above is not used in our proof. 
The assertion about $\varphi^{(s)}(0)=0$ 
and $s\geq 1$  ensures that $T_\xi\bdy\Omega_\varphi$ has order of contact with $\bdy\Omega_\varphi$ at
$\xi$ greater than $2$, for any $\xi$ of the form $(0,\xi_2,\dots,\xi_n)$ such that 
$(\xi_2,\dots,\xi_n)\in\mathbb{S}^{n-1}$, the unit sphere in $\mathbb{C}^{n-1}$. 
In other words, condition $(c)$ states that
$\Omega_\varphi$ is weakly pseudoconvex.    
\end{remark}
We now provide
\begin{proof}[The proof of Proposition~\ref{P:examples-convex}]
First we will prove that $\Omega_\varphi$ is $B$-regular. If we were to appeal to
Proposition~\ref{P:C-strict_convex}, we would have to compute $T_{\xi}\bdy\Omega_{\varphi}$ for each
$\xi\in \bdy\Omega_{\varphi}$, which is
somewhat unpleasant. Instead, since $\Omega_\varphi$ is a bounded, convex, Reinhardt domain in
$\Cn$, by \cite[Proposition~3.1]{dieudunghung:brcd05} it is enough to show that $\Omega_\varphi$ does not have 
any non-constant analytic disc in $\bdy\Omega_\varphi$. If possible, let $\Psi:\D\lrarw\Cn$ be a non-constant
holomorphic map such that $\Psi(\D)\subseteq\bdy\Omega_\varphi$. Let $\Psi:=(\psi_1,\dots,\psi_n)$, 
where $\psi_i:\D\lrarw\C$ are holomorphic maps for $1\leq i\leq n$. Since $\Psi(\D)\subseteq\bdy\Omega_\varphi$, 
for every $\zeta\in\D$ we have
\[
  \varphi\big(\psi_1(\zeta)\overline{\psi_1(\zeta)}\big)+\sum_{j=2}^{n} |\psi_j(\zeta)|^2=1.
\]
Applying the operator $\bdy^2/\bdy\zeta\bdy\overline\zeta$ to both sides we get
\begin{align}\label{E:examples-convex-derivative-0}
  \varphi''(|\psi_1|^2)\bigg|\frac{\overline{\bdy\psi_1}}
  {\bdy \zeta}\psi_1\bigg|^2+\varphi'(\big|\psi_1|^2)\bigg|\frac{\bdy\psi_1}
  {\bdy \zeta}\bigg|^2+\sum_{j=2}^{n}\bigg|\frac{\bdy\psi_j}{\bdy \zeta}\bigg|^2\equiv 0\quad \text{on $\D$}.
\end{align}
From conditions $(a)$ and $(d)$ it follows that $\varphi'(x)\geq 0$ and $\varphi''(x)\geq 0$ for all
$x\geq 0$. Therefore, from \eqref{E:examples-convex-derivative-0}, we have 
\[
  \frac{\bdy\psi_j}{\bdy\zeta}\equiv 0 \quad \forall j=2,\dots,n.
\]
Therefore, $\psi_j$ is constant for every $j=2,\dots,n$. Since $\Psi$ is non-constant, $\psi_1$ must be 
non-constant. In particular, $\psi_1\not\equiv 0$ on $\D$. Thus, there exists a non-empty open set
$U\subseteq\D$ such that $|\psi_1(\zeta)|>0$ for all $\zeta\in U$. 
From \eqref{E:examples-convex-derivative-0}, and as $\varphi'(x)>0$ for $x>0$,
\[
  \bigg|\frac{\bdy\psi_1}{\bdy\zeta}\bigg|\equiv 0 \quad \text{on $U$}.
\]
Therefore, $\psi_1$ is constant on $U$. Hence, by the Identity Principle, $\psi_1$ is constant
on $\D$, which contradicts the fact that $\Psi$ is non-constant. Hence, $\Omega_\varphi$ is $B$-regular.
\smallskip

Fix $l\in\{2,\dots,n\}$. Let $\varepsilon\in (-1,-1/2)$.
Hence, $\delta_{\Omega_\varphi}(z_\varepsilon)= 1+\varepsilon$.
Fix $\lambda\in\D$. Define $w_{\lambda}:=z_\varepsilon +\big(\varphi^{-1}(\delta_{\Omega_\varphi}
(z_\varepsilon))\big)^{1/2}\lambda\bas_1=\big((\varphi^{-1}(\delta_{\Omega_\varphi}
(z_\varepsilon)))^{1/2}\lambda,0,\dots,0,\varepsilon,0,\dots,0\big)$.
Since $\varphi$ is convex and $\varphi(0)=0$, 
it follows that 
\begin{align*}
  \varphi\big(\varphi^{-1}(\delta_{\Omega_\varphi}(z_\varepsilon))|\lambda|^2\big)+\varepsilon^2
  &\leq |\lambda|^2\varphi\big(\varphi^{-1}(\delta_{\Omega_\varphi}(z_\varepsilon))\big)+\varepsilon^2\\
  &\leq \delta_{\Omega_\varphi}(z_\varepsilon)+\varepsilon^2\leq 1+\varepsilon+\varepsilon^2<1.
\end{align*}
Therefore, $w_{\lambda}\in\Omega_\varphi$ for all $\lambda\in\D$. Hence,
for all
$\varepsilon\in(-1,-1/2)$ 
it follows by definition that 
\begin{align}\label{E:lower-bound-r}
\big(\varphi^{-1}(\delta_{\Omega_\varphi}(z_\varepsilon))\big)^{1/2}\leq\rdsc{\Omega_\varphi}
(z_\varepsilon;\bas_1). 
\end{align}
Now, fix $\lambda=1/\sqrt{2}$ and consider the point
\[
  \widehat w:= z_\varepsilon +\big(2\varphi^{-1}(2\delta_{\Omega_\varphi}(z_\varepsilon))\big)^{1/2}
  (1/\sqrt{2})\bas_1
  =\big((\varphi^{-1}(2\delta_{\Omega_\varphi}(z_\varepsilon)))^{1/2},0,\dots,0,\varepsilon,0,\dots,0\big).
\]
We compute:
\begin{align*}
  \varphi\big(\varphi^{-1}(2\delta_{\Omega_\varphi}(z_\varepsilon))\big)+\varepsilon^2&=2\delta_{\Omega_\varphi}
  (z_\varepsilon)+\varepsilon^2\\&=2(1+\varepsilon)+\varepsilon^2=1+(1+\varepsilon)^2>1 
\end{align*}
By our definition of $\Omega_\varphi$, $z_\varepsilon +\big(2\varphi^{-1}(2\delta_{\Omega_\varphi}(z_\varepsilon))\big)^{1/2}
  (1/\sqrt{2})\bas_1\notin\Omega_\varphi$. Therefore, by definition we have 
$\rdsc{\Omega_\varphi}(z_\varepsilon;\bas_1)\leq\big(2\varphi^{-1}(2\delta_{\Omega_\varphi}
(z_\varepsilon))\big)^{1/2}$. By condition $(d)$, 
$\varphi^{-1}$ is concave. Hence, from the fact that $\varphi(0)=0$, it follows that $\varphi^{-1}(2t)\leq
2\varphi^{-1}(t)$ for all $t\geq 0$.
Therefore, for all $\varepsilon\in(-1,-1/2)$ 
\begin{align}\label{E:upper-bound-r}
\rdsc{\Omega_\varphi}(z_\varepsilon;\bas_1)\leq\big(2\varphi^{-1}(2\delta_{\Omega_\varphi}
(z_\varepsilon))\big)^{1/2}\leq 2\big(\varphi^{-1}(\delta_{\Omega_\varphi}(z_\varepsilon))\big)^{1/2}.
\end{align}
Hence, combining \eqref{E:lower-bound-r} and \eqref{E:upper-bound-r}, part~$(i)$ follows.
\smallskip

Therefore, by Result~\ref{R:graham}, for every
  $\varepsilon\in (-1,-1/2)$ we have
\[
  \frac{\|v\|}{4\big(\varphi^{-1}(\delta_{\Omega_\varphi}(z_\varepsilon))\big)^{1/2}}\leq 
  k_{\Omega_\varphi}(z_\varepsilon;v)
  \leq
  \frac{\|v\|}{\big(\varphi^{-1}(\delta_{\Omega_\varphi}(z_\varepsilon))\big)^{1/2}}\quad\forall v\in{\rm span}_\C\{\bas_1\}. 
  \]
\end{proof}

Having established the above proposition, we can present a family of domains that illustrate how abundant
domains satisfying the conditions of Theorem~\ref{T:monge-ampere_result} are. This substantiates the discussion
right after Remark~\ref{rem:B-reg_geom}.

\begin{example}\label{Eg:monge-ampere-not-holder}
An example demonstrating that there exists a sequence of domains $\Omega_n$, $n\geq 2$, satisfying all the
conditions stated in Theorem~\ref{T:monge-ampere_result}.  
\end{example}

\noindent Let us define the function $\varphi:[0,+\infty)\lrarw[0,+\infty)$ as
$$\varphi(x)\coloneqq \begin{cases}
  0, & \text{if }x=0,\\
  e^2\exp(-1/x)/3,&\text{if }0<x<1/2,\\
  (4x-1)/3,  &\text{otherwise}.\end{cases}$$
The conditions $(a)$--$(c)$ in Proposition~\ref{P:examples-convex}, with $k=2$ and $s=2$, follow from elementary computations and that
$\varphi$ is strictly increasing is almost self-evident. Note that the function $x\longmapsto e^2\exp(-1/x)/3$ is not
convex on the whole of $(0,+\infty)$. One restricts the latter function to $(0,1/2)$ to ensure
convexity of $\varphi$. Now, we will show that $\varphi$ is convex on $(0,+\infty)$. It is easy to check that
$\varphi$ is $\smoo^2$ on $(0,+\infty)$. A computation gives
$$\varphi''(x)= \begin{cases}
  e^2\exp(-1/x)(1-2x)/3x^4, & \text{if } 0<x<1/2,\\
  0, &\text{otherwise.}
  
\end{cases}$$
This implies that $\varphi''\geq 0$, and hence, $\varphi$ is convex.
\smallskip

Now, set $\Omega_n:=\Big\{z\in\Cn: \varphi(z_1\overline z_1)+\sum_{2\leq j\leq n} |z_j|^2<1\Big\}$.
Since $\varphi(1)=1$, we have $(1,0,\dots,0)\in\bdy\Omega_n$. From the definition of $\varphi$, it follows that in a
neighbourhood of $(1,0,\dots,0)$, $\bdy\Omega_n$ is given by
\[
  |z_1|=\sqrt{1-\Big(3\sum_{2\leq j\leq n} |z_j|^2/4\Big)}.
\]
Therefore, $\bdy\Omega_n$ is $\smoo^\infty$-smooth in a neighbourhood of $(1,0,\dots,0)$. Since $\Omega_n$ is
Reinhardt, the above holds true in a neighbourhood of any $(e^{i\theta},0,\dots,0)\in\bdy\Omega_n$, where
$\theta\in[0,2\pi)$. This, together with the conclusions of the previous paragraph
implies that $\bdy\Omega_n$ is $\smoo^2$-smooth. Writing $z_1 = x_1+ iy_1$,
the real Hessian of $z_1\longmapsto \varphi(z_1\zbar_1)$ is
\[
  \begin{bmatrix}
      \,4x^2_1\varphi''(z_1\zbar_1) 
      + 2\varphi'(z_1\zbar_1) & 4x_1y_1\varphi''(z_1\zbar_1) \\
      4x_1y_1\varphi''(z_1\zbar_1) &
      4y^2_1\varphi''(z_1\zbar_1) 
      + 2\varphi'(z_1\zbar_1)\,
  \end{bmatrix}.
\]      
Thus, the real Hessian of the function
$z\longmapsto \varphi(z_1\overline z_1) + \sum_{2\leq j\leq n} |z_j|^2$ is
positive semi-definite, due to which the second fundamental form of $\bdy\Omega$
is non-negative at each $\xi\in \bdy\Omega_n$. Thus, the condition $(e)$ holds true
as well (with $a=1$).
\smallskip

Since all the conditions stated in Proposition~\ref{P:examples-convex} are satisfied, we need to
establish one last condition. Let
\[
  z_{\nu} := (0, -1+ 1/(\nu + 2), 0,\dots, 0) \quad\text{and}
  \quad \univ_{\nu} := \bas_1,
\]
for $\nu = 1, 2, 3,\dots$\,. By Proposition~\ref{P:examples-convex}, we have
\[
  \rdsc{\Omega_n}(z_{\nu}; \univ_{\nu}) \geq
  \frac{1}{\big(\log(e^{2}/3\delta_{\Omega_n}(z_{\nu}))\big)^{1/2}}
\]
for $\nu = 1, 2, 3,\dots$; clearly, \eqref{E:eta-ratio_infty} holds true. Therefore, $\Omega_n$ satisfies all
the conditions stated in Theorem~\ref{T:monge-ampere_result}. \hfill$\blacktriangleleft$
\smallskip

We now provide an example that relates to the discussion right after Theorem~\ref{T:proper-map}. In
particular, our example demonstrates that Theorem~\ref{T:proper-map} is \textbf{not} vacuously true.
This is therefore an appropriate place to present a pair of definitions that were deferred to a later
section in our discussion in Section~\ref{SS:extension}.

\begin{definition}\label{D:SHL-domain}
Let $\Omega\varsubsetneq \Cn$ be a bounded domain in $\Cn$.
\begin{itemize}[leftmargin=25pt]
  \item[$(a)$] (Charabati, \cite{charabati:hrscmae15}) We say that $\Omega$ is a \emph{strongly hyperconvex Lipschitz domain} if there exists a
  neighbourhood $U$ of $\overline{\Omega}$ and a Lipschitz plurisubharmonic function $\rho: U\lrarw \R$ such that: 
  \begin{itemize}
      \item[$(1)$] $\rho < 0$ in $\Omega$ and $\bdy\Omega = \{z\in U: \rho(z) = 0\}$.
      \item[$(2)$] There exists a constant $c>0$ such that $dd^c\rho \geq c\omega_n$ on $\Omega$ in the sense of
      currents.
  \end{itemize}
  Here, $\omega_n$ denotes the standard K{\"a}hler form, $(i/2)\sum_{j=1}^n dz_j\wedge d\zbar_j$, on $\Cn$. \smallskip
  
  \item[$(b)$] We say that $\Omega$ is a \emph{regular strongly hyperconvex Lipschitz domain} if $\Omega$ is
  a strongly hyperconvex Lipschitz domain and if, writing $dd^c\rho = \sum\nolimits_{\mu,\nu=1}^n
  b_{\mu\overline{\nu}}\,dz_{\mu}\wedge d\zbar_{\nu}$, each $b_{\mu\overline{\nu}}\in 
  \leb{\infty}(\Omega)$.
\end{itemize}
\end{definition}

\begin{example}\label{Eg:proper-map}
An example demonstrating that there exist domains $D\varsubsetneq \mathbb{C}^m$, $\Omega\varsubsetneq \Cn$,
$m<n$, and a proper holomorphic map
$F:D\lrarw\Omega$ such that
\begin{itemize}
  \item $D$, $\Omega$, and $F$ satisfy the conditions of Theorem~\ref{T:proper-map},
  \item $D$, $\Omega$ have non-smooth boundaries.
\end{itemize}
\end{example}

\noindent Express $\zeta\in\C$ or $z_2$, where $(z_1, z_2)\in \C^2$,
as $u+iv$, $u,v\in\R$, and define $h(\zeta)=2u^2-\beta(v)+v^4$, where 
$$\beta(v)\coloneqq \begin{cases}
  v^2, & \text{if }v\geq 0,\\
  -v^2,&\text{if }v<0.\end{cases}$$
Therefore,
\begin{align}
  h(\zeta) &= 2u^2+(v^2-1/2)^2-1/4 \; \;\text{whenever $v\geq 0$}, \label{E:sum_of_sq} \\
  h(\zeta) &\geq -1/4 \quad \forall \zeta\in \C. \label{E:h(z)}
\end{align}
Let us define 
\begin{align*}
  D&:=\{(z_1,z_2)\in\mathbb{C}^2:|z_1|^2+h(z_2)<0\},\\
  \Omega&:=\{(w_1,w_2,w_3)\in\phi(D)\times\C:|w^2_1-1|^2+h(w_2)+|w_3|^2<0\},
\end{align*}
where $\phi(z_1,z_2):=(\sqrt{z_1+1},z_2)$, $(z_1,z_2)\in D$, and where $\sqrt{\bcdot}$ denotes 
the principal branch of the square root. We need to argue that $\phi\in\hol(D;\C^2)$. 
We will show that $D$ is bounded, a consequence of which will be that $\phi\in\hol(D;\C^2)$. 
Define the continuous functions $\rho_1:\C^2\lrarw\R$ and $\rho_2:\mathbb{C}^3\lrarw\R$ by
\[
  \rho_1(z_1,z_2)=|z_1|^2+h(z_2),\quad\rho_2(z_1,z_2,z_3)=|z^2_1-1|^2+h(z_2)+|z_3|^2.
\]
By definition, $(z_1,z_2)\in \overline{D}$ implies that $\im(z_2)\nless 0$. Therefore, 
$D\subseteq\{(z_1,z_2)\in\mathbb{C}^2:|z_1|^2+2(\re (z_2))^2+((\im(z_2))^2-1/2)^2<1/4\}$ and, hence, $D$ is a 
bounded domain in $\mathbb{C}^2$. Moreover, $(z_1,z_2)\in D$ implies that $|z_1|<1/2$.
Therefore, $\phi$ is holomorphic on $D$. 
\smallskip

It is easy to compute that for every $(u,v)\in\R^2$,
\begin{equation}\label{E:w_partials}
  \frac{\bdy^2h}{\bdy u^2}(u,v)=4 \quad\text{and}\quad
  \frac{\bdy^2h}{\bdy v^2}(u,v)\,=\,
  \begin{cases}
    -2+12v^2, & \text{if }v\geq 0,\\
    2+12v^2, &\text{if }v<0.
  \end{cases}
\end{equation}
Note that $\rho_1|_{\{(z_1, z_2)\,:\,\im(z_2)>0\}}$ and
$\rho_1|_{\{(z_1, z_2)\,:\,\im(z_2)<0\}}$ are of class $\smoo^2$. From this and from \eqref{E:w_partials},
\[
  h|_{\{u+iv\,:\,v<0\}}\in {\sf sh}(\{u+iv: v<0\}) \quad\text{and}
  \quad h|_{\{u+iv\,:\,v>0\}}\in {\sf sh}(\{u+iv: v>0\}).
\]
Thus, if we fix $u_0+iv_0\in \C$ such that $v_0\neq 0$, then the averages of $h$ on  circles with centre
$u_0+iv_0$ and radii in $(0, |v_0|)$ dominate $h(u_0+iv_0)$.
Now observe that
\begin{align*}
  \frac1{2\pi}\int_0^{2\pi}h(u_0+re^{i\tht})\,d\tht
  = \frac1{2\pi}\int_0^{2\pi}\big(2(u_0+r\cos\tht)^2 + (r\sin\tht)^4\big)\,d\tht
  &\geq 2u_0^2 \notag \\
  &= h(u_0) 
\end{align*}
for each $u_0\in \R$ and $r>0$.
From the last two observations, we conclude that $h\in {\sf sh}(\C)$.
From the above discussion, we establish three of the conditions for $D$ to be a strongly hyperconvex Lipschitz
domain, the third being a consequence of our last assertion (about $h$) and the definition of $\rho_1$: 
\begin{itemize}
  \item $\rho_1$ is (uniformly) Lipschitz on any relatively compact domain in $\C^2$.
  \item $\bdy D=\rho_1^{-1}\{0\}$.
  \item $\rho_1\in {\sf psh}(\C^2)$.
\end{itemize}
Moreover, as $(z_1,z_2)\in D$ implies that $\im(z_2)> 0$, \eqref{E:w_partials} implies, after an
easy calculation, that:
\begin{itemize}
  \item For every $(z_1, z_2)\in D$,
  \[
    dd^c{\rho_1}(z_1,z_2) = (i/2)\left(dz_1\wedge d\zbar_1 + \big(1/2 +
    3(\im(z_2))^2\big)dz_2\wedge d\zbar_2\right)
    \geq (1/2)\omega_{2}(z_1, z_2).
  \]
\end{itemize}
It remains to show that $D$ is connected. To this end, observe that, by the definition of
$D$, it suffices to show that
\[
  \Delta := \{(u, v)\in \R^2 : h(u,v)<0\}
\]
is connected. By the definition of $h$, if $(u,v)\in \bdy\Delta$, then $v\geq 0$. Thus, by \eqref{E:sum_of_sq},
$(u,v)\in \bdy\Delta$ implies that
\begin{align*}
  &\big(\sqrt{2}u, v^2-(1/2)\big) \text{ lies on the circle with centre $0$ and radius $1/2$} \\
  \iff &\big(2\sqrt{2}u, 2(v^2-(1/2))\big) \text{ lies on the circle with centre $0$ and radius $1$}.
\end{align*}
Thus, we have the continuous, bijective map $\psi: \mathbb{S}^1\lrarw \bdy\Delta$:
\[
  \psi(s,t) := \left(\frac{s}{2\sqrt{2}},\,\sqrt{\frac{t+1}2}\right), \quad (s,t)\in \mathbb{S}^1.
\]
As $\mathbb{S}^1$ is compact, $\psi$ is a homeomorphism; thus
$\bdy\Delta$ is a Jordan curve. It follows from the Jordan Curve Theorem that $\Delta$ is
connected. By our remarks preceding the definition of $\Delta$, we conclude that $D$ is a domain.
Together with the four bullet-points earlier in this paragraph, it follows that $D$ is a
strongly hyperconvex Lipschitz domain.
\smallskip

It is elementry to see that the only point
$p\in\mathbb{C}^2$ such that $\nabla\rho_1(p)=0$ and $p\in\bdy D$ is $p=(0,0)$. Therefore, $D$ has a non-smooth
boundary. Since $\bdy D$ is $\smoo^\infty$-smooth 
in a small neighbourhood of $\xi$ for each $\xi\in\bdy D\setminus\{(0,0)\}$, we only need to examine $\bdy D$ around
$(0,0)$.
Recall that if $(z_1,z_2)\in\overline{D}$, for $z_2=u+iv$, then $v\geq 0$. Thus, we will solve the equation
\begin{align*}
  |z_1|^2+2u^2-v^2+v^4&=0,\quad\text{and}\\
  v&\geq 0,
\end{align*}
explicitly for $v$ in terms of $z_1$ and $u$ for $\|(z_1,u)\|$ small. Solving the auxiliary quadratic
equation
\[
  X^2-X+(2u^2+|z_1|^2)=0,
\]
we find that there exists a $\delta>0$ sufficiently small that, if $\Gamma_\delta$ denotes the connected 
component of 
$\bdy D\cap B^3(0,\delta)\times\R$ containing $(0,0)$, then
\[
  \Gamma_\delta: v=\sqrt{|z_1|^2+2u^2}+O(\|(z_1,u)\|^2)\quad\text{$\forall (z_1,u)\in B^3(0,\delta)$}.
\]
By our previous remarks about the points in $\bdy D\setminus\{(0,0)\}$, $\bdy D$ is a Lipschitz manifold. \emph{We
have shown that $D$ satisfies all the conditions of Theorem~\ref{T:proper-map}}.
\smallskip

Since $D$ is connected and $\phi$ is continuous, $\Omega$ is connected. By definition,
$(w_1,w_2,w_3)\in\overline\Omega$ implies that $\im(w_2)\nless 0$. Hence, 
$\Omega\subseteq\{(w_1,w_2,w_3)\in\mathbb{C}^3:|w_1^2-1|^2+2(\re (w_2))^2+((\im(w_2))^2-
1/2)^2+|w_3|^2<1/4\}$. Therefore, $\Omega$ is a bounded domain in $\C^3$. We have shown that
$h\in {\sf sh}(\C)$; by definition
it follows that $\rho_2\in {\sf psh}(\C^3)$.
Hence, from the above discussion, and a few easy estimates, we have: 
\begin{itemize}
  \item $\rho_2$ is (uniformly) Lipschitz on any relatively compact domain in $\C^3$.
  \item For every $(w_1,w_2,w_3)\in \Omega$,
  \begin{align}
    dd^c&\rho_2(w_1,w_2,w_3) \notag \\
    &= (i/2)\left(4|w_1|^2dw_1\wedge d\overline{w}_1+\big(1/2+3(\im(w_2))^2\big)dw_2\wedge d\overline{w}_2+dw_3\wedge d\overline{w}_3\right)\notag\\
  &\geq (1/2) \omega_{3}(w_1, w_2, w_3).\label{E:ddcrho2}
  \end{align}
  \end{itemize}
Therefore, since $\Omega$ is bounded, $dd^c\rho_2\vert_\Omega$ is a $(1,1)$-current with $\leb{\infty}(\Omega)$-coefficients. 
Now, define 
\[
  U:= \text{the connected component of 
  $\{(w_1,w_2,w_3)\in\C^3:\rho_2(w_1,w_2,w_3)<1/2\}$ containing $\overline{\Omega}$.}
\]
Note that if $(w_1, w_2, w_3)\in \rho_{2}^{-1}\{0\}$, then $(-w_1, w_2, w_3)\in 
\rho_{2}^{-1}\{0\}$. However, we claim that $\left(\rho_2|_{U}\right)^{-1}\{0\} = \bdy\Omega$.
To see this: note that by \eqref{E:h(z)}, $\rho_2(w_1, w_2, w_3)<1/2$
implies that $|w_1|^2\geq 1-|w_1^2-1|\geq(1-\sqrt{3}/2)>0$. Therefore, 
\begin{equation}\label{E:u-disjoint-0}
  \{(0,w_2,w_3):(w_2,w_3)\in\C^2\}\cap\overline{U}=\emptyset.
\end{equation}
From \eqref{E:u-disjoint-0}, it easily follows that $\big(\rho_2^{-1}\{0\}\cap U,\,
\rho_2^{-1}\{0\}\cap (\C^3\setminus U)\big)$ is a separation of $\rho_{2}^{-1}\{0\}$. Therefore:
\begin{itemize}
  \item $\left(\rho_2|_{U}\right)^{-1}\{0\} = \bdy\Omega$.
\end{itemize}
From the three bullet-points in this paragraph, we conclude that 
$\Omega$ is a regular strongly hyperconvex Lipschitz domain. 
\smallskip

Again, it is elementry to see that the only point 
$q\in\C^3$ such that $\nabla\rho_2(q)=0$ and $q\in\bdy\Omega$ is $q=(1,0,0)$. Therefore, $\Omega$ has a non-smooth
boundary and in a manner similar to our previous argument for $D$, we can conclude that $\bdy\Omega$ is a Lipschitz
manifold. \emph{We have now shown that $\Omega$ satisfies all the conditions of Theorem~\ref{T:proper-map}}.
\smallskip

Finally, let us define
\[
  F(z_1,z_2):=\Big(\sqrt{z_1+1},z_2,0\Big)\quad\text{$\forall (z_1,z_2)\in D$}.
\]
Since $\phi$ is holomorphic, $F:D\lrarw\Omega$ is holomorphic. Now, to prove that $F:D\lrarw\Omega$ is proper, it
suffices to show that $\phi=(\phi_1,\phi_2):D\lrarw\phi(D)$ is proper. To this end, consider
$\big((z_{1,\,\nu}, z_{2,\,\nu})\big)_{\nu\geq 1}\subset D$ such that 
$(z_{1,\,\nu}, z_{2,\,\nu})\lrarw\xi$ for some $\xi\in\bdy D$.
This implies that 
\begin{align*}
  |z_{1,\,\nu}|^2+h(z_{2,\,\nu})&\lrarw 0\\
  \implies |(\phi_1(z_{1,\,\nu}, z_{2,\,\nu}))^2-1|^2+h(\phi_2(z_{1,\,\nu},z_{2,\,\nu}))&\lrarw 0, 
\end{align*}
as $\nu\to\infty$. Therefore, since $\phi(D)\subseteq\{(w_1,w_2)\in\C^2:|w^2_1-1|^2+h(w_2)<0\}$, it
follows that $\big(\phi(z_{1,\,\nu}, z_{2,\,\nu})\big)_{\nu\geq 1}$ exits every compact in $\phi(D)$. Hence,
$\phi:D\lrarw\phi(D)$ is proper. Let $F=(F_1,F_2,F_3)$. Since $(z_1,z_2)\in D$ implies that $|z_1|<1/2,$ it follows that
$\frac{\bdy F_1}{\bdy z_1}=\frac1{2\sqrt{z_1+1}}\in\leb{\infty}(D)$. Therefore, clearly
\[
  \pd{F_{\mu}}{z_{j}}{}\,\overline{\pd{F_{\nu}}{z_{k}}{}} \in \leb{\infty}(D)
\]
for each $j, k: 1\leq j, k\leq 2$
and each $\mu, \nu: 1\leq \mu, \nu\leq 3$. Hence, $D$, $\Omega$, and $F$ satisfy the conditions of
Theorem~\ref{T:proper-map} such that $D$ and $\Omega$ have non-smooth boundaries. \hfill$\blacktriangleleft$
\medskip

\section{Analytical preliminaries}\label{S:ana}
This section is devoted to several important observations and results related, primarily, to the proofs of
the theorems in Section~\ref{SS:extension}. 

\subsection{Concerning the complex Monge--Amp{\`e}re equation}\label{SS:M-A_prelims}
We begin with a brief discussion on $B$-regular domains, introduced in Section~\ref{S:intro} and whose
definition was deferred to this section.

\begin{definition}
Let $\Omega$ be a bounded domain in $\Cn$. We say that $\Omega$ is $B$-regular if $\bdy\Omega$ is a $B$-regular
set: i.e., $\bdy\Omega$ has the property that each function $\varphi\in \smoo(\bdy\Omega; \R)$ is the uniform
limit on $\bdy\Omega$ of a sequence $(u_{\nu})_{\nu\geq 1}$ of continuous plurisubharmonic functions defined
on open neighbourhoods of $\bdy\Omega$ (each such neighbourhood depending on the function $u_{\nu}$).
\end{definition}

As mentioned in Section~\ref{S:intro}, \cite[Theorem~4.1]{blocki:cmaohd96} by B{\l}ocki establishes that
the Dirichlet problem \eqref{E:monge-ampere} admits a unique solution for any non-negative $f\in
\smoo(\overline\Omega; \R)$ and any $\varphi\in \smoo(\bdy\Omega; \R)$. Note that when
$u|_{\Omega}\notin \smoo^2(\Omega; \R)$, the left-hand side of the equation in \eqref{E:monge-ampere} is
interpreted as a current of bidegree $(n, n)$. That this makes sense when $u\in {\rm psh}(\Omega)\cap
\smoo(\overline\Omega)$ was established by Bedford--Taylor \cite{bedfordtaylor:dpcmae76}, who established an
existence and uniqueness theorem for \eqref{E:monge-ampere} with the above-mentioned data for strongly
pseudoconvex domains. Furthermore, \cite[Theorem~4.1]{blocki:cmaohd96} is a consequence of Theorem~8.3
in \cite{bedfordtaylor:dpcmae76}.
\smallskip

As hinted at in our discussion in Section~\ref{SS:extension}, for proving Theorem~\ref{T:proper-map}:
\begin{itemize}
  \item We need an existence theorem for the Dirichlet problem \eqref{E:monge-ampere} in which the datum
  $f$ is far less well-behaved than what is mentioned right after \eqref{E:monge-ampere} (or in the previous
  paragraph).
  \item With the data just mentioned, we also require some information on the modulus of continuity of
  the solutions to \eqref{E:monge-ampere}
\end{itemize}
The earliest results to provide information of the above-mentioned type are presented in
\cite{guedjkolodziejzeriahi:hcsmae08}. The result (which relies strongly on the latter work) that we need
for our proofs is the following. 

\begin{result}[paraphrasing {\cite[Theorem~2]{charabati:hrscmae15}} by Charabati]
\label{R:charabati}
Let $\Omega\varsubsetneq \Cn$ be a strongly hyperconvex Lipschitz domain. Let $f\in
\leb{p}(D, \Leb_{2n})$ for some $p>1$ and let $\varphi: \bdy\Omega\lrarw \R$ be the restriction of
a function defined on a neighbourhood of $\bdy\Omega$ and of class $\smoo^{1,1}$. Then, the Dirichlet
problem \eqref{E:monge-ampere} has a unique solution $u$, and $u$ belongs to
$\smoo^{0,s}(\overline{\Omega})$ for each $s\in (0, 1/(nq+1))$, where $1/p+1/q=1$.
\end{result}

In the above result, in saying that a function defined on an open set $\mathscr{O}\subseteq \Cn$ is of class
$\smoo^{1,1}$ we mean that it is of class $\smoo^1$ and all it first-order partial derivatives satisfy a
Lipschitz condition.

\subsection{Miscellaneous analytical results}
We shall now present several supporting results that we shall need in our proofs. We begin with a technical result.
In the proof that follows, if $\mathscr{O}$ is an open set in $\C^m$, then $\distr(\mathscr{O}; \C)$ will denote
the space of all $\C$-valued distributions on $\mathscr{O}$.

\begin{proposition}\label{P:L^r}
Let $D$ be a bounded domain in $\C^m$, $\Omega$ a bounded domain in $\Cn$, $m<n$, and let
$F: D\lrarw \Omega$ be a proper holomorphic map. Let $\rho\in \smoo(\Omega)\cap {\sf psh}(\Omega)$
and assume $\rho$ is such that, writing
\[
  dd^c\rho = (i/2)\sum\nolimits_{\mu,\nu=1}^n b_{\mu\overline{\nu}}\,dw_{\mu}\wedge d\wbar_{\nu},
\]
$b_{\mu\overline{\nu}} \in \leb{\infty}(\Omega)$ for each $\mu,\nu: 1\leq \mu,\nu\leq n$. Then:
\begin{itemize}[leftmargin=25pt]
  \item[$(a)$] In the sense of currents
  \[
    dd^c(\rho\circ F) = (i/2)\sum\nolimits_{j,k=1}^m
    \left(\sum\nolimits_{\mu,\nu=1}^n (b_{\mu\overline{\nu}}\circ F)\,
    \pd{F_{\mu}}{z_j}{}\,\overline{\pd{F_{\nu}}{z_k}{}}\,\right)dz_j\wedge d\zbar_{k}.
  \]
  \item[$(b)$] Suppose $F$ satisfies the condition in Theorem~\ref{T:proper-map}.
  If we write $(dd^c(\rho\circ F))^m = \widetilde{f}d\beta_m$, then $\widetilde{f}$ is a
  function that is non-negative a.e. and is of class $\leb{p/m}(D, \Leb_{2m})$, $p$ as in Theorem~\ref{T:proper-map}.
\end{itemize}
\end{proposition}
\begin{proof}
Let $\chi \in\smoo_{c}^\infty(\Cn)$ be a non-negative cut-off function with $\|\chi\|_1=1$ and
support in $\D^n$ that satisfies
$\chi(w) = \chi(|w_1|,\dots, |w_n|)$ for all $w\in \Cn$. Define $\chi_{\eps} := \eps^{-2n}\chi(\bcdot/\eps)$.
For $\eps>0$, write $\Omega_{\eps} := \{w\in \Omega: \delta_{\Omega}(w) < \sqrt{n}\eps\}$. Note that
$\rho*\chi_{\eps}$
is defined on $\Omega_{\eps}$. It is well-known that
$\rho*\chi_{\eps} \in \smoo^\infty(\Omega_{\eps})\cap {\sf psh}(\Omega_{\eps})$, that 
\begin{equation}\label{E:convo_dec}
  \rho*\chi_{\eps}(w) \searrow \rho(w) \; \text{as $\eps\searrow 0$} \quad \forall w\in \Omega,
\end{equation}
and, since $\rho\in \smoo(\Omega)$, this convergence is uniform on compact subsets of $\Omega$; see,
for instance, \cite[Chapter~2]{klimek:pp91}. Define $D^\eps := F^{-1}(\Omega_{\eps})$. Note that
$D^{\eps}\Subset D$, since $F$ is proper, and that $(\rho*\chi_{\eps})\circ F$ is well-defined on $D^{\eps}$.
Thus, we shall abbreviate $(\rho*\chi_{\eps})\circ (F|_{D^\eps})$ as $(\rho*\chi_{\eps})\circ F$ and understand 
that this is defined on $D^\eps$.
Moreover, if we fix $\eps_0 > 0$, then $(\rho*\chi_{\eps})\circ F\in \distr(D^{\eps_0}; \C)$ for every
$\eps\in (0, \eps_0)$. Now, if $\tst \in \mathscr{D}(D^{\eps_0}; \C)$, i.e., a test function on
$D^{\eps_0}$, then, as ${\rm supp}(\eta)$ is a compact subset of $D^{\eps_0}$, by the observation
following \eqref{E:convo_dec}, we deduce that
\[
  (\rho*\chi_{\eps})\circ F \lrarw (\rho\circ F)|_{D^{\eps_0}} \; \; 
  \text{in $\distr(D^{\eps_0}; \C)$},
\]
and the above is true for all $\eps_0>0$. This implies, since differentiation is a continuous operator
on $\distr(D^{\eps_0}; \C)$, that
\[
  \spd{z_j,\zbar_k}{2}((\rho*\chi_{\eps})\circ F) \lrarw 
  \spd{z_j,\zbar_k}{2}\big((\rho\circ F)|_{D^{\eps_0}}\big) \; \; 
  \text{in $\distr(D^{\eps_0}; \C)$},
\]
for all $j, k: 1\leq j, k\leq m$, and the above is true for all $\eps_0>0$. Finally, since
$\eps_0 > 0$ is arbitrary, we have
\begin{equation}\label{E:convo_UCT}
  \left\langle \spd{z_j,\zbar_k}{2}((\rho*\chi_{\eps})\circ F), \tst\right\rangle
  \lrarw \left\langle\spd{z_j,\zbar_k}{2}(\rho\circ F), \tst\right\rangle \; \; 
  \text{as $\eps\searrow 0$}
\end{equation}
for all $j, k: 1\leq j, k\leq m$, and for any test function $\tst\in \mathscr{D}(D; \C)$.
\smallskip

On the other hand, as $(\rho*\chi_{\eps})\circ F \in \smoo^\infty(D^{\eps})$, we have
\begin{align}
  \spd{z_j,\zbar_k}{2}((\rho*\chi_{\eps})\circ F)(z)
  &= \left\langle\rho,\,(\spd{z_j,\zbar_k}{2}\,\chi_{\eps})(F(z)-\bcdot)\right\rangle \notag \\
  &= \sum_{\mu, \nu=1}^n \left\langle \rho,\,
     \pd{\chi_{\eps}}{w_\mu\wbar_{\nu}}{2}(F(z)-\bcdot)\,
     \pd{F_{\mu}}{z_j}{}(z)\,\overline{\pd{F_{\nu}} {z_k}{}}(z)\right\rangle \notag \\
  &= \sum_{\mu, \nu=1}^n \left\langle \spd{w_{\mu},\wbar_{\nu}}{2}\rho,\,
     \chi_{\eps}(F(z)-\bcdot)\right\rangle
     \pd{F_{\mu}}{z_j}{}(z)\,\overline{\pd{F_{\nu}}{z_k}{}}(z) \notag \\
  &= \sum_{\mu, \nu=1}^n \big(b_{\mu\overline{\nu}}*\chi_{\eps}\big)(F(z))\,
     \pd{F_{\mu}}{z_j}{}(z)\,\overline{\pd{F_{\nu}}{z_k}{}}(z)  \quad \forall z\in D^{\eps}. \label{E:convo_link}
\end{align}
The second equality above is due to the chain rule while the last equality follows from the definition of
the convolution and from our hypothesis. Note that the above is true for all $\eps>0$ and for
every $j, k: 1\leq j,k\leq m$.
Since $b_{\mu\overline{\nu}} \in \leb{\infty}(\Omega)$ for each $\mu, \nu: 1\leq \mu, \nu\leq n$
and as $\Omega$, being bounded, has finite Lebesgue measure,
$b_{\mu\overline{\nu}} \in \leb{1}(\Omega, \Leb_{2n})$ for each $\mu, \nu: 1\leq \mu, \nu\leq n$. Thus,
a diagonal argument gives us a sequence $(l_i)_{i\geq 1}\subset \N$ such that
\[
  \big(b_{\mu\overline{\nu}}*\chi_{1/l_i}\big)(F(z))\lrarw b_{\mu\overline{\nu}}(F(z))
  \; \; \text{for a.e. $z\in D$ as $i\to \infty$}, 
\]
and the above holds true for each $\mu, \nu: 1\leq \mu, \nu\leq n$. Now fix a test function
$\tst\in \mathscr{D}(D; \C)$. Then, as $(\chi_{\eps} )_{\eps>0}$ is an approximation of the
identity and as ${\rm supp}(\tst)\Subset D$, there exists a number
$C(\eta, j, k)>0$ such that
\[
  \|((b_{\mu\overline{\nu}}*\chi_{1/l_i})\circ F)\,\spd{z_j}{}\!F\,\overline{\spd{z_k}{}\!F}\,\tst\|_1 \leq
  C(\eta, j, k)\|b_{\mu\overline{\nu}}\|_{\infty}\|\eta\|_{\infty} \quad \forall i\in \N.
\]
We can thus apply the Dominated Convergence Theorem, which gives
\begin{equation}\label{E:convo_DCT}
  \left\langle ((b_{\mu\overline{\nu}}*\chi_{1/l_i})\circ F)\,\spd{z_j}{}\!F\,\overline{\spd{z_k}{}\!F}, \tst\right\rangle
  \lrarw \langle (b_{\mu\overline{\nu}}\circ F)\,\spd{z_j}{}\!F\,\overline{\spd{z_k}{}\!F}, \tst\rangle \; \; \text{as $i\to \infty$},
\end{equation}
and which holds true for each $\mu, \nu: 1\leq \mu, \nu\leq n$. From \eqref{E:convo_UCT},
\eqref{E:convo_link}, and \eqref{E:convo_DCT}, part~$(a)$ follows.
\smallskip

Let us write
\[
  dd^c(\rho\circ F) = (i/2)\sum\nolimits_{j,k=1}^m a_{j\overline{k}}\,dz_j\wedge d\zbar_{k},
\]
Let $\mathscr{E}_{\mu\overline{\nu}}$ denote the set on which either $b_{\mu\overline{\nu}}$ is undefined or
$|b_{\mu\overline{\nu}}|$ equals
$+\infty$. Then, by hypothesis, $\Leb_{2n}(\mathscr{E}_{\mu\overline{\nu}})=0$ for each $\mu, \nu: 
1\leq \mu, \nu\leq n$. Since $F$ is proper, its critical set has zero Lebesgue measure. It is well-known
that, owing to the latter property,
\[
 \Leb_{2m}(F^{-1}(\mathscr{E}_{\mu\overline{\nu}})) = 0 \quad
 \forall \mu, \nu: 1\leq \mu, \nu\leq n.
\]
Thus, the set on which either $b_{\mu\overline{\nu}}\circ F$ is undefined or $|b_{\mu\overline{\nu}}\circ F|$
equals $+\infty$ has zero Lebesgue measure for each
$\mu, \nu: 1\leq \mu, \nu\leq n$. Since each $b_{\mu\overline{\nu}}$ is essentially bounded,
$b_{\mu\overline{\nu}}\circ F \in \leb{\infty}(D)$. Then, by part~$(a)$ and by our assumption on
$F$, we have
\[
a_{j\overline{k}}\in \leb{p}(D, \Leb_{2m}) \subseteq \leb{m}(D, \Leb_{2m})
\]
for each $j,k: 1\leq j,k\leq m$. By \cite[Proposition~2.7]{{bedfordtaylor:dpcmae76}}, there exists a constant
$C_m\neq 0$ such that $\widetilde{f}$ is given by
\begin{equation}\label{E:f-tilde}
  \widetilde{f} = C_m\!\!\sum_{\sigma\in S_m}{\rm sign}(\sigma)\prod_{j=1}^m a_{j\,\overline{\sigma(j)}}\,,
\end{equation}
where $S_m$ denotes the group of permutations of $m$ objects. 
By part~$(a)$, $\widetilde{f}$ is a homogeneous polynomial of degree $m$ whose indeterminates are
$\spd{z_j}{}\!F_{\mu}\,\overline{\spd{z_k}{}\!F_{\nu}}$, $1\leq j, k\leq m, 1\leq \mu, \nu\leq n$, and
whose coefficients are $\leb{\infty}$-functions. Thus, by H{\"o}lder's inequality and our assumption on $F$,
$\wt{f}\in \leb{p/m}(D, \Leb_{2m})$. Finally, since $\rho\circ F\in {\sf psh}(D)$, 
$(dd^c(\rho\circ F))^m$ is a positive $(m,m)$-current and so, in view of \eqref{E:f-tilde},
$\widetilde{f}$ is non-negative a.e.
\end{proof}

The next two results will be vital to proving Theorems~\ref{T:monge-ampere_result}
and~\ref{T:complex-geodesic}, respectively. Below, and in subsequent sections, $(\cHess{\varphi})$ will denote the
complex Hessian of $\varphi$ and $\langle\bcdot\,, \bcdot\rangle$ the standard Hermitian inner product on $\Cn$.

\begin{proposition}\label{P:monge-ampere}
Let $\Omega\varsubsetneq \Cn$ be a $B$-regular domain, and let $f\in \smoo(\overline{\Omega}; \R)$
be a non-negative function. Let $\varphi: \mathcal{U}\lrarw \R$ be a $\smoo^\infty$-smooth function, where
$\mathcal{U}$ is a neighbourhood of $\overline{\Omega}$. Assume that $\varphi\in {\sf psh}(\mathcal{U})$ and
that there exists a constant $\eps > 0$ such that
$\langle v, (\cHess{\varphi})(z) v \rangle \geq \eps\|v\|^2$
for every $z\in \Omega$ and $v\in \Cn$. Let $\omega_{f,\,\varphi}$ be a modulus of continuity of the
unique solution to the Dirichlet problem
\begin{equation}\label{E:monge-ampere_aux}
  \left.
  \begin{array}{r l}
    (dd^c{u})^n &\mkern-9mu{= f\beta_n, \; 
    \text{ $u\in \smoo(\overline{\Omega})\cap {\sf psh}(\Omega)$},} \\
    u|_{\bdy\Omega} &\mkern-9mu{= -2\varphi|_{\bdy\Omega},}
    \end{array} \right\}
\end{equation}
for the complex Monge--Amp{\`e}re equation. Then, there exists a constant $c > 0$ such that
\[
  k_{\Omega}(z; v) \geq 
  c\,\sqrt{\eps}\,\frac{\|v\|}
            {\big(\omega_{f,\,\varphi}(\delta_{\Omega}(z))\big)^{1/2}}
\]
for every $z\in \Omega$ and $v \in \Cn$.
\end{proposition}
\begin{proof}
Let $u_{f,\,\varphi}$ denote the unique solution to the Dirichlet problem \eqref{E:monge-ampere_aux}, which is
guaranteed by the fact that $\Omega$ is $B$-regular. Define
\[
  {\Phi}_1(z)\,:=\,u_{f,\,\varphi}(z)+\varphi(z) \quad \forall z\in \overline\Omega.
\]
By a version of Richberg's regularisation theorem (see \cite[Theorem~1.1]{blocki:cmaohd96}, for instance), there
exists a $\smoo^\infty$-smooth plurisubharmonic function $\Phi_2$ on $\Omega$ such that
\begin{equation} \label{E:richberg}
  0\leq \Phi_2(z)- \Phi_1(z)\leq \omega_{f,\,\varphi}(\delta_{\Omega}(z)) \quad
  \forall z\in \Omega.
\end{equation}
Clearly, $\Phi_2$ extends continuously to $\overline\Omega$ (we shall refer to this extension as $\Phi_2$ as
well) and
\begin{equation} \label{E:Phi_2_bdy}
  \Phi_2(z) = -\varphi(z) \quad \forall z\in \bdy\Omega.
\end{equation}
Write $\auxu(z) := \Phi_2(z)+\varphi(z)$ for each $z\in \overline\Omega$.
Since $\Phi_2$ is plurisubharmonic on $\Omega$, by our condition on $\cHess{\varphi}$,
\begin{equation} \label{E:U_Levi_pos}
  \langle v, (\cHess{\auxu})(z) v \rangle \geq \eps\|v\|^2 \quad \forall z\in \Omega
  \text{ and } \forall v\in \Cn.    
\end{equation}

Fix $z\in \Omega$. As $\bdy\Omega$ is compact, there exists a point $\xi_z\in \bdy\Omega$ such that
$\delta_{\Omega}(z) = \|z - \xi_z\|$. It follows from \eqref{E:richberg} that
\begin{align}
  |\auxu(z)|\,&\leq\,|\Phi_2(z) - \Phi_1(z)| + |\Phi_1(z) + \varphi(z)| \notag \\
  &\leq\,\omega_{f,\,\varphi}(\delta_{\Omega}(z)) +
  |(\Phi_1(z) + \varphi(z))- (\Phi_1(\xi_z) + \varphi(\xi_z))|. 
  \label{E:midway}
\end{align}
Since $\varphi$ is $\smoo^\infty$-smooth on $\mathcal{U}\supseteq \overline{\Omega}$, hence Lipschitz on
$\overline{\Omega}$, there exists a constant $C_1 >0 $ such that
\[
  |(\Phi_1(z) + \varphi(z))- (\Phi_1(\xi_z) + \varphi(\xi_z))|
  \leq C_1\omega_{f,\,\varphi}(\delta_{\Omega}(z)).
\]
Here, we have used the fact that $\|z-\xi_z\| = \delta_{\Omega}(z)$ and that $\omega_{f,\,\varphi}$ is a modulus of
continuity of $\auxu$. Combining the last estimate with
\eqref{E:midway}, we get
\[
  |\auxu(z)| \leq (1+C_1)\,\omega_{f,\,\varphi}(\delta_{\Omega}(z)).
\]
The above holds true for each $z\in \Omega$ as $z$ was chosen arbitrarily and as the choice of $C_1$ depends only
on $\omega_{f,\,\varphi}$.
\smallskip

By \eqref{E:Phi_2_bdy}, we have $\auxu|_{\bdy\Omega} = 0$. Thus, by the maximum principle, $\auxu$ is a smooth negative
plurisubharmonic function on $\Omega$. Thus, from the last inequality, \eqref{E:U_Levi_pos}, and
Result~\ref{r:Kob_low_bd_Sibony}, we conclude that
\[
  k_{\Omega}(z;v) \geq \left( \frac{1}{(1+C_1)\alpha} \right)^{1/2}\sqrt{\eps}\,\frac{\|v\|}
  {\big(\omega_{f,\,\varphi} (\delta_{\Omega}(z))\big)^{1/2}} \quad \forall z\in \Omega \text{ and } \forall v \in \Cn,
\]
which is the desired lower bound. 
\end{proof}

\begin{result}[Bharali, {\cite[Proposition~2.1]{bharali:cgtbrhltl16}}]\label{r:bharali}
\label{r:hardy-littlewood-type}
Let $\varphi:[0,r_0)\lrarw[0,+\infty]$ be a function of class of $\leb{1}\left([0,r_0), \Leb_1\right)$ for some $r_0\in(0,1)$. Let $g\in\hol(\D)$ and assume that 
\[
  |g'(re^{i\theta})|\leq\varphi(1-r) \quad\forall r: 1-r_0<r<1 \; \;\text{and} \; \;\forall\theta\in\R.
\]
Then, $g$ extends continuously to $\bdy\D$.
\end{result}
\smallskip

\section{Geometric preliminaries}\label{S:geom}
In this section, we present a few definitions and a key result that plays a vital role in the proof of
Theorem~\ref{T:complex-geodesic}. 

\subsection{Definitions on the geometry of domains} We begin with a definition that was deferred to a later
section in our discussion in Section~\ref{SS:extension}.

\begin{definition}\label{D:lipschitz}
The boundary of a domain $\Omega\varsubsetneq \Cn $ is called a \emph{Lipschitz manifold} if, for each
$p\in \bdy{\Omega}$, there exists a
neighbourhood $\Nb_p$ of $p$, a unitary map $\unitary_p$, and a $\R^{2n-1}$-open neighbourhood of the origin
that is the domain of a Lipschitz function $\varphi_{p}$ such that, denoting the affine map
$z \longmapsto \unitary_p(z-p)$ by $\unitary^p$ and writing 
$\varz = (\varz_1,\dots, \varz_n):= \unitary^p(z)$, we have
\begin{align*}
  \unitary^p(\Nb_p\cap \Omega) &= \{(\varz', \varz_n)\in \unitary^p(\Nb_p):
    \im(\varz_n) > \varphi_p(\varz', \re(\varz_n)), \ (\varz', \re(\varz_n))\in {\sf dom}(\varphi_p)\},
    \; \text{and} \\
    \unitary^p(\Nb_p\cap \bdy{\Omega}) &= \{(\varz', \varz_n)\in \C^{n-1} \times \C:
    \im(\varz_n) = \varphi_p(\varz', \re(\varz_n)), \ (\varz', \re(\varz_n))\in {\sf dom}(\varphi_p)\}.
\end{align*}
\end{definition}

The above is a familiar class of domains in real analysis, with the difference that the
domains encountered classically are domains in $\R^N$. Thus, for domains in $\Cn$,
the unitary changes of coordinates mentioned
in Definition~\ref{D:lipschitz} replace the orthogonal changes of coordinates in the classical
definition. Such domains have many useful properties and have been studied extensively (see, e.g.,
\cite{adams:Ss75} and the references therein).
\smallskip

An \emph{open right circular cone with aperture $\theta$} is an open subset of $\Cn$ of the form
\[
 \{z\in \Cn : \re[\,\langle z, v\rangle\,] > \cos(\theta/2)\|z\|\}
 =: \Gamma(\theta, v),
\]
where $v$ is some unit vector in $\Cn$, $\theta\in (0, \pi)$ (and $\langle\bcdot\,,\,\bcdot\rangle$ is as introduced
in Section~\ref{S:ana}).
\smallskip

Now, we present another definition that is needed in proving Theorem~\ref{T:proper-map}.

\begin{definition}\label{D:unif_cone_condn}
Let $\Omega\varsubsetneq \Cn$ be a domain. We say that $\Omega$ satisfies an \emph{interior-cone
condition with aperture $\theta$} if there exist constants  $r_0 > 0$, $\theta\in (0, \pi)$,
and a compact subset $K\subset \Omega$ such that, for each $z\in \Omega\setminus K$, there exist a point 
$\xi_z\in \bdy{\Omega}$ and a unit vector $v_z$ such that
\begin{itemize}
 \item $z$ lies on the axis of the cone $\xi_z+\Gamma(\theta, v_z)$, and
 \item $(\xi_z+\Gamma(\theta, v_z))\cap \mathbb{B}^n(\xi_z, r_0) \subset \Omega$.
\end{itemize}
We say that $\Omega$ satisfies a \emph{uniform interior-cone condition} if there exists a $\theta\in
(0, \pi)$ such that $\Omega$ satisfies an interior-cone condition with aperture $\theta$.
\end{definition}

The property defined above is a part of the hypothesis of the next result, which is a generalisation of the
classical Hopf Lemma for plurisubharmonic functions.

\begin{result}[Mercer, {\cite[Proposition~1.4]{mercer:gHlphmcdCd93}}] \label{R:hopf-lemma}
Let $\Omega$ be a bounded domain in $\Cn$ that satisfies an interior-cone condition with aperture
$\theta$, $\theta\in (0, \pi)$. Let $\rho: \Omega\lrarw [-\infty, 0)$ be a plurisubharmonic function.
Then, there exists a constant $c>0$ such that
\[
  \rho(z)\leq -c(\delta_{\Omega}(z))^{\alpha}  
\]
$($where $\alpha = \pi/\theta)$ for every $z\in \Omega$.
\end{result}
\smallskip

\subsection{Essential propositions} We begin this section with a result that, through its role in the proof of
Proposition~\ref{P:monge-ampere}, is at the core of several of our proofs.

\begin{result}[paraphrasing {\cite[Proposition~6]{sibony:1981}}]
\label{r:Kob_low_bd_Sibony}
Let $\Omega \subset \Cn$ be a domain. There exists a constant $\alpha > 0$ $($which is a universal
constant$)$ such that if $z\in \Omega$ and if there exists a negative plurisubharmonic function $u$ on $\Omega$ that
is of class $\mathcal{C}^2$ in a neighbourhood of $z$ and satisfies 
\[
  \langle v, (\cHess{u})(z) v \rangle \geq c \|v\|^2 \quad
  \forall v \in \Cn,
\]
for some $c > 0$, then 
\[
  k_{\Omega}(z;v) \geq \Big( \frac{c}{\alpha} \Big)^{1/2}\frac{\|v\|}{|u(z)|^{1/2}}
  \quad \forall v \in \Cn.
\]
\end{result}

Our next result is crucial to the proof of Theorem~\ref{T:complex-geodesic}. A few remarks are in order here. 
From \cite[Theorem~1.7]{blocki:cmaohd96} (which generalises a result by Sibony in \cite{sibony:cdp87} to domains 
with non-$\smoo^1$ boundaries), a domain $\Omega\Subset\Cn$ is $B$-regular if and only if $\Omega$ admits a 
plurisubharmonic peak function at each $p\in\bdy\Omega$: i.e., a function $u_p:
\overline\Omega\lrarw(-\infty,0]$ 
belonging to $\smoo(\overline\Omega)\cap{\sf psh}(\Omega)$ satisfying 
\[
  u_p(z)<0\;\;\forall z\in\overline\Omega\setminus\{p\}\quad\text{and}\quad u_p(p)=0.
\]
Since a bounded $\C$-strictly convex domain is convex, its $B$-regularity may seem obvious. But recall that
\begin{itemize}
    \item a convex domain $\Omega$ merely admits a holomorphic \textbf{weak} peak function at each
    $p\in \bdy\Omega$ (weak peak functions are not germane to our discussion, so we shall
    skip the definition).
    \item the polydisc $\D^n$, although convex, is not $B$-regular when $n\geq 2$. 
\end{itemize}
We will prove Proposition~\ref{P:C-strict_convex}
by showing that bounded $\C$-strictly convex domains admit a holomorphic peak function at
each boundary point.
Since this proposition is vital to the proof of Theorem~\ref{T:complex-geodesic}, we provide a proof of it here. 
(A word of caution: $\C$-strictly convex domains are not to be confused with $\C$-convex domains.)

\begin{proposition}\label{P:C-strict_convex}
Bounded $\C$-strictly convex domains are $B$-regular.    
\end{proposition}

\begin{proof}
Let $\Omega$ be a bounded $\C$-strictly convex domain. Let $p\in\bdy\Omega$. By definition, there exists a
support hyperplane of
$\Omega$, say ${\mathcal{H}}$, containing $p$ such that the $\C$-affine hyperplane
\[
  \widetilde {\mathcal{H}}:=p+\big((\mathcal{H}-p)\cap i(\mathcal{H}-p)\big)
\]
satisfies $\widetilde {\mathcal{H}}\cap\overline\Omega=\{p\}$. Let us write $H=\widetilde {\mathcal{H}}-p$. Let
$v\in H^{\bot}$, where the orthogonal complement is with respect to the standard Hermitian inner product
$\langle\bcdot,\bcdot\rangle$, on $\Cn$, and let $\|v\|=1$. Let $L:=\{p+\zeta v: \zeta\in\C\}$. Let
${\sf proj}_L$ denote the orthogonal projection onto the $\C$-\textbf{affine} line $L$. To
clarify: as $p\in L$, ${\sf proj}_{L}(p) = p$. Next, define
\[
  \omega:=\{\zeta\in\C: p+\zeta v\in 
{\sf proj}_L(\Omega)\},
\]
which is a bounded domain in $\C$. Now define $\pi:\Cn\lrarw \C$ by 
$\pi(z):=\langle {\sf proj}_L(z)-p,v\rangle$. Let $z\in\overline\Omega$ and ${\sf proj}_L(z)=p+\zeta v$, for
some $\zeta\in\C$. Then, $\zeta\in\overline\omega$. 
Therefore, $\pi(z)=\langle{\sf proj}_L(z)-
p,v\rangle=\langle\zeta v,v\rangle=\zeta\in\overline\omega$. Hence,
$\pi(\overline\Omega)\subseteq\overline\omega$. Clearly,
$\pi\vert_\Omega:\Omega\lrarw\omega$ is holomorphic. Also, note that, $\pi(p)=0$ and
$\widetilde{\mathcal{H}}=\pi^{-1}\{0\}$.
Let $(p_j)_{j\geq 1}\subset\Omega$ be such that $p_j\to p$. Since $\pi$ is continuous, $\pi(p_j)\to\pi(p)=0$.
Now $\pi(p_j)\in\omega$ but $\pi(p)=0\notin\omega$ since $\pi^{-1}\{0\} = 
\widetilde{\mathcal{H}}$ is disjoint from $\Omega$. Therefore, $0\in\bdy\omega$. 
\smallskip

Consider two distinct points $x,y\in\omega$. Then, there exist $\widehat x,\widehat y\in\Omega$ such
that $\pi(\widehat x)=x$ and $\pi(\widehat y)=y$. Let $\mathscr{H}_{x,y}:=
\pi^{-1}(\{tx+(1-t)y:t\in[0,1]\})$, which is convex. Since $\Omega\cap\mathscr{H}_{x,y}$ is convex and
$\widehat x,\widehat y\in\Omega\cap\mathscr{H}_{x,y}$, it easily follows that 
$\{tx+(1-t)y:t\in[0,1]\}\subset\omega$. Since $x$ and $y$ were arbitrarily chosen, $\omega$ is convex.
Hence, $\omega$ is a bounded simply connected domain in $\C$ such that $\bdy\omega$ is a Jordan curve
and such that $0\in\bdy\omega$. Therefore, by the Riemann Mapping Theorem and Carath{\'e}odory's extension
theorem, there
exists a homeomorphism $\psi:\overline\omega\lrarw\overline\D$ such that
$\psi\vert_\omega:\omega\lrarw\D$ is a biholomorphism, $\psi(\bdy\omega)=\bdy\D$, and such that
$\psi(0)=1$. 
\smallskip

Let $f$ be a non-vanishing holomorphic peak function for $\D$ at $1$. Let
$g:=f\circ\psi\circ\pi|_{\overline\Omega}$. Then, $g\vert_\Omega$ is holomorphic. Let
$z\in\overline\Omega\setminus\{p\}$. Since 
$\widetilde{\mathcal{H}}=\pi^{-1}\{0\}$ and $\widetilde {\mathcal{H}}\cap\overline\Omega=\{p\}$,
$\pi(z)\in\overline\omega\setminus\{0\}$. Hence, 
$\psi(\pi(z))\neq 1$. Therefore, as $f$ is a holomorphic peak function at $1$, $|g(z)|<1$. Clearly,
$g(p)=1$. So, $g$ is a holomorphic peak function for $\Omega$ at $p$. 
Define the function $u:\overline\Omega\lrarw (-\infty,0]$ by $u(z):=\log|g(z)|$. Clearly, $u$ is
continuous, $u\vert_\Omega$ is plurisubharmonic, $u(z)<0$ for $z\in\overline\Omega\setminus\{p\}$, and 
$u(p)=0$. Thus, $u$ is a plurisubharmonic peak function for $\Omega$ at $p$. As
$p\in \bdy\Omega$ was arbitrarily chosen, by the remarks involving \cite[Theorem~1.7]{blocki:cmaohd96}
preceding this proof, the result follows.
\end{proof}
\smallskip

\section{The proof of Theorem~\ref{T:monge-ampere_result}}
\textbf{Fix} some $f\in\smoo(\overline\Omega; \R)$ such that $f$ is non-negative. Since $\Omega$ is $B$-regular,
there exists a unique plurisubharmonic solution to the equation~\eqref{E:monge-ampere} for the boundary data
$-2\varphi|_{\bdy\Omega}: \bdy\Omega\lrarw \R$, where $\varphi$ is a $\smoo^\infty$-smooth function
defined on some neighbourhood $\mathcal{U}_{\varphi}$ of $\overline{\Omega}$ such that
\begin{itemize}
  \item $\varphi\in {\sf psh}(\mathcal{U}_{\varphi})$, and
  \item $\langle v, (\cHess{\varphi})(z) v \rangle \geq \eps_{\varphi}\|v\|^2$, for some
  $\eps_{\varphi} > 0$, for every $z\in \Omega$ and $v\in \Cn$.
\end{itemize}
Let us denote the solution by $u_{\varphi}$ (for simplicity of notation, as $f$ is fixed). If possible, let 
$u_{\varphi}$ belong to the class $\smoo^{0,\alpha}(\overline\Omega)$ for some $\alpha\in(0,1]$. Write
$\omega(r):=r^\alpha$ for $r\in[0,\infty)$. Then, by assumption, there exists a constant $C_{\varphi}>0$ such that
for all $z_1,z_2\in\overline\Omega$
\[
  |u_{\varphi}(z_1)-u_{\varphi}(z_2)|\leq C_{\varphi}\,\omega(\|z_1-z_2\|).
\]
By Proposition~\ref{P:monge-ampere}, we conclude that there exists a constant $M_{\varphi}>0$ such that 
\begin{equation}\label{E:low_bound}
  k_\Omega(z;v)\geq \frac{M_{\varphi}}{\omega(\delta_\Omega(z))^{1/2}}
  =\frac{M_{\varphi}}{(\delta_\Omega(z))^{\alpha/2}}\quad 
  \text{$\forall z\in\Omega$ and $\forall v\in\Cn: \|v\|=1$}.
\end{equation}
By the contractivity of the affine embeddings 
$\D\ni\zeta\longmapsto z_{\nu}+\big(\rdsc{\Omega}(z_{\nu}; \univ_{\nu})\zeta\big)\univ_{\nu}$ for the
Kobayashi metric, we have
\[
  k_\Omega(z_{\nu};\univ_{\nu})\leq \frac1{\rdsc{\Omega}(z_{\nu}; \univ_{\nu})}\quad
  \forall \nu = 1, 2, 3,\dots\,.
\]
Combining this inequality with \eqref{E:low_bound}, it follows that
\[
  \frac{\rdsc{\Omega}(z_{\nu}; \univ_{\nu})}{\big(\delta_\Omega(z_{\nu})\big)^{\alpha/2}}\leq \frac1{M_{\varphi}}\quad 
  \forall \nu = 1, 2, 3,\dots,
\]
which contradicts \eqref{E:eta-ratio_infty}. Thus, we have produced a large class of boundary data with the
properties stated in Theorem~\ref{T:monge-ampere_result} for
which the solution to \eqref{E:monge-ampere} is not in $\smoo^{0,\alpha}(\overline{\Omega})$ for any
$\alpha\in (0,1]$. Hence the result. \hfill $\qed$
\medskip

\section{The proofs of Theorem~\ref{T:proper-map} and Theorem~\ref{T:proper-map_str-pseud}}
\label{S:proper-map_proof}
To prove the above-mentioned theorem, we will need the following lemma:

\begin{lemma}\label{L:lips-bdy-ineq}
Let $D\varsubsetneq\C^m$ be a bounded domain such that $\bdy D$ is a Lipschitz manifold. Let $p\in\bdy D$. Then, in
the notation of Definition~\ref{D:lipschitz}, but with $\psi_p$ in place of $\varphi_p$,
there exist a neighbourhood $V$ of $p$, $V\Subset\Nb_p$, and a constant
$C>1$ $($both depending on $p)$ such that
\[
  \delta_D(z)\leq\vary({\unitary}^p(z))\leq C\delta_D(z)\quad\forall z\in V\cap D,
\]
where $\vary(\varz',\varz_m):=\im(\varz_m)-\psi_p(\varz,\re(\varz_m))$ for $(\varz',\varz_m)\in
\unitary^p(V\cap D)$.
\end{lemma}

\begin{proof}
The first inequality is obvious. So, we will prove the second inequality. There exists a neighbourhood $V$ of $p$,
$V\Subset\Nb_p$ such that ${\rm diam}(V)<{\rm dist}(\overline{V},\C^m\setminus\Nb_p)$. Due to the latter
inequality, for every $z\in V\cap D$, there exists $x_z\in\Nb_p\cap\bdy D$ such that 
\[
  \delta_D(z)=\delta_{\Nb_p\cap D}(z)=\|z-x_z\|.
\]
Since $\psi_p$ is Lipschitz, $\vary$ is Lipschitz; let $C$ denote its Lipschitz constant. It follows that for
every $z\in V\cap D$
\[
  \vary(\unitary^p(z))=\|\vary(\unitary^p(z))-\vary(\unitary^p(x_z))\|\leq C\|\unitary^p(z)-\unitary^p(x_z)\|=
  C\|z-x_z\|=C\delta_D(z).
\]
Hence, the result follows.
\end{proof}

We are now in a position to give a proof of Theorem~\ref{T:proper-map}.

\begin{proof}[The proof of Theorem~\ref{T:proper-map}]
Fixing $p\in\bdy D$, it suffices to show that $F$ extends continuously to a
map $\widetilde{F}_p$ on $(U_p\cap\bdy D)\cup D$, for some neighbourhood $U_p$ of $p$. This is
because, if we are able to show this, then the expression
$$
  \widetilde{F}(z) := \begin{cases}
                        F(z), &\text{if $z\in D$}, \\
                        \widetilde{F}_p(z), &\text{if $z\in U_p\cap\bdy D$},
                      \end{cases}
$$
would be well-defined. This is because if $p, q\in \bdy{D}$ are distinct points and 
$(U_p\cap U_q)\cap \bdy{D}\neq
\emptyset$, then for each $\xi\in (U_p\cap U_q)\cap \bdy{D}$, we would have
\[
  \widetilde{F}_p(\xi) = \lim_{D\ni z\to \xi}\widetilde{F}_p(z) = \lim_{D\ni z\to \xi}F(z)
   = \lim_{D\ni z\to \xi}\widetilde{F}_q(z) = \widetilde{F}_q(\xi).
\]
We will establish the above objective in the following three steps. 
\medskip

\noindent{{\textbf{Step~1.}} \emph{A preliminary estimate for $\|F'(z)v\|$ for every $z\in D$, every $v\in\C^m.$}}
\smallskip

\noindent{Since $\Omega$ is strongly hyperconvex Lipschitz, by
Result~\ref{R:charabati}, there exists a
unique plurisubharmonic solution to the Dirichlet problem
\begin{align*}
  (dd^c{u})^n &=0, \; 
  \text{ $u\in \smoo(\overline{\Omega})\cap {\sf psh}(\Omega)$}, \notag \\
  u|_{\bdy\Omega} &=-2\|w\|^2, \label{E:monge-ampere}
\end{align*}
for the complex Monge--Amp{\`e}re equation belonging to
$\smoo^{0,s}(\overline{\Omega})$ for some $s\in (0,1)$. Let us denote the solution by $\widetilde u$.
Therefore, there exists a constant $M_0>0$ such that for all 
$w_1,w_2\in\overline\Omega$
\[
  |\widetilde u(w_1)-\widetilde u(w_2)|\leq M_0\|w_1-w_2\|^s.
\]
By Proposition~\ref{P:monge-ampere}, we conclude that there exists a constant ${M}>0$ such that 
\begin{equation}\label{E:low_bound-theorem}
  k_\Omega(w;v)\geq M\frac{\|v\|}{(\delta_\Omega(w))^{s/2}}\quad\forall w\in \Omega,\quad\forall v\in\Cn.
\end{equation}
By the contractivity of the inclusion map
$\mathbb{B}^m(z,\delta_D(z))\hookrightarrow D$ for the
Kobayashi metric, we have
\begin{equation}\label{E:upper-bound-kobayashi}
 k_D(z;v)\leq\frac{\|v\|}{\delta_D(z)}\quad \forall z\in D,\quad\forall v\in\C^m.   
\end{equation}
Therefore, combining \eqref{E:low_bound-theorem} and \eqref{E:upper-bound-kobayashi}, for every $z\in D$, every
$v\in\C^m$, we have 
\begin{align}\label{E:F'(z)-upper-bound}
  \frac{{M}\|F'(z)v\|}{\big(\delta_\Omega(F(z))\big)^{s/2}}&\leq k_\Omega(F(z);F'(z)v)\leq k_D(z;v)\leq\frac{\|v\|}
  {\delta_D(z)}\notag\\\
  \implies\|F'(z)v\|&\leq \frac{1}{{M}}\frac{\|v\|}{\delta_D(z)}\big(\delta_\Omega(F(z))\big)^{s/2}. 
\end{align}
}

\medskip

\noindent{{\textbf{Step~2.}} \emph{A bound for $\delta_\Omega(F(z))$ in \eqref{E:F'(z)-upper-bound}.}}
\smallskip

\noindent{Since $\Omega$ is a regular strongly hyperconvex Lipschitz domain, there exists a Lipschitz
continuous plurisubharmonic function
$\rho$ on a neighbourhood $\Omega'$ of $\overline\Omega$ such that $\rho<0$ on $\Omega$,
$\rho\vert_{\bdy\Omega}\equiv0$, and satisfies all other conditions in Definition~\ref{D:SHL-domain}-$(b)$.
Now, $\Omega$ satisfies a uniform interior-cone condition. This is a consequence of $\bdy\Omega$
being a Lipschitz manifold. Establishing this involves a standard argument and, so, we shall just
briefly outline the steps involved:
\begin{itemize}
  \item If we fix $q\in \bdy\Omega$, then, in the notation of Definition~\ref{D:lipschitz}, to each
  $\varw \in \unitary^q(\mathcal{V}_q\cap \Omega)$, let us associate the point
  \[
    \xi_{\varw} := \big(\varw', \re(\varw_n)+i\varphi_{q}(\varw', \re(\varw_n))\big)\in
    \unitary^{q}(\mathcal{V}_q\cap \bdy\Omega),
  \]
  where $\varw$ will denote points in $\unitary^{q}(\mathcal{V}_q)$ (hence, in a small neighbourhood of
  $0$) and\linebreak
  $\varw =: (\varw', \varw_n)$.
  \item There exists a small neighbourhood $W_q$ of $0$, $W_q\Subset \unitary^{q}(\mathcal{V}_q)$, and
  constants $s_q>0$ and $\theta_q\in (0, \pi)$ such that the truncated cone
  \[
    \big(\xi_{\varw}+\Gamma(\theta_q, (0,\dots,0,i))\big)\cap \mathbb{B}^n(\xi_{\varw}, s_q) \subset
    \unitary^{q}(\mathcal{V}_q\cap \Omega) 
  \]
  for each $\varw\in W_q\cap \unitary^{q}(\Omega)$, where $\theta_q$ is determined by the
  Lipschitz constant of $\varphi_q$.
  \item By compactness of $\bdy\Omega$, there exist $q_1,\dots, q_N\in \bdy\Omega$ such
  that $\bigcup_{j=1}^N(\unitary^{q_j})^{-1}(W_{q_j})\supset \bdy\Omega$ and, for each $x\in K$, one can
  produce the (truncated) cones that satisfy the condition stated in
  Definition~\ref{D:unif_cone_condn}, for
  \[
    K = \Omega\setminus \bigcup_{1\leq j\leq N}(\unitary^{q_j})^{-1}(W_{q_j}),
  \]
  with the parameters $\theta = \min\{\theta_{q_1},\dots, \theta_{q_N}\}$ and
  $r_0 = \min\{s_{q_1},\dots, s_{q_N}\}$.
\end{itemize}
As $\Omega$ satisfies a uniform interior-cone condition, by
Result~\ref{R:hopf-lemma}, there exist constants $c_0>0$ and $\alpha>1$ such that 
\begin{equation}\label{E:hopf-lemma}
 \rho(w)\leq -c_0(\delta_\Omega(w))^\alpha\quad\forall w\in\Omega.   
\end{equation}
Since $\rho$ is continuous and plurisubharmonic on $\Omega$, 
$\rho\circ F:D\lrarw(-\infty,0)$ is a continuous plurisubharmonic function. We will
show that $\rho\circ F$ extends continuously to $\overline D$. Let $p'\in\bdy D$ and
$(x_\nu)_{\nu\geq 1}\subset D$ be a sequence such that $x_\nu\lrarw p'$ as $\nu\to\infty$. Let
$(x_{\nu{_k}})_{k\geq 1}$ be an arbitrary subsequence of $(x_\nu)_{\nu\geq 1}$. Since $F$ is proper and
$\Omega$ is bounded, there exist a subsequence $(x_{\nu_{k_{l}}})_{l\geq 1}$ and $q'\in\bdy\Omega$ such 
that $F(x_{\nu_{k_{l}}})\lrarw q'$ as $l\to\infty$. Therefore, since $\rho(q')=0$, 
$\rho\circ F(x_{\nu_{k_{l}}})\lrarw 0$ as $l\to\infty$. 
Hence, every subsequence of $(\rho\circ F(x_\nu))_{\nu\geq 1}$ has a subsequence
that converges to $0$, whence we conclude that $\rho\circ F(x_\nu)\lrarw 0$ as $\nu\to\infty$.  
Thus, the function 
$$\widetilde \rho(z)\coloneqq \begin{cases}
  \rho\circ F(z), & \text{if }z\in D,\\
  0,&\text{if }z\in\bdy D.\end{cases}$$
is a continuous function on $\overline D$.
\smallskip

Since $\rho\circ F$ is a continuous plurisubharmonic function, $(dd^c(\rho\circ F))^m$ is defined in the sense of
currents. Write:
\[
  (dd^c(\rho\circ F))^m=\widetilde f\beta_m.
\]
By Proposition~\ref{P:L^r}, we see that $\widetilde{f}$ is non-negative a.e. and that
$\widetilde{f}\in\leb{p/m}(D, \Leb_{2m}) $. Therefore, since $p/m>1$ and since $D$ is strongly hyperconvex Lipschitz, by Result~\ref{R:charabati}, there exists a
unique plurisubharmonic solution to the Dirichlet problem
\begin{align*}
  (dd^c{u})^m &=\widetilde f\beta_m, \; 
  \text{ $u\in \smoo(\overline{D})\cap {\sf psh}(D)$}, \notag \\
  u|_{\bdy D} &=0,
\end{align*}
for the complex Monge--Amp{\`e}re equation and this belongs to
$\smoo^{0,s_0}(\overline{D})$ for some $s_0\in (0,1)$. Since, by the above discussion,
$\widetilde\rho$ is a solution to the above Dirichlet problem, by
uniqueness, it follows that $\widetilde{\rho}\in\smoo^{0,s_0}(\overline{D})$. By the $\smoo^{0,s_0}$ property, 
there exists a constant $C_0>0$ such that for every $z\in D$ we have 
\begin{align}
 -\rho(F(z))= -\widetilde\rho(z)=|\widetilde\rho(z)|&\leq C_0(\delta_D(z))^{s_0}\notag\\
  \implies \delta_\Omega(F(z))&\leq C_1(\delta_D(z))^{s_*},\label{E:bdy-om-bdy-D-ineq}
\end{align}
where $C_1:=(C_0/c_0)^{1/\alpha}>0$, $s_*:=s_0/\alpha\in(0,1)$, and the last inequality follows from \eqref{E:hopf-lemma}. 
}
\medskip
\pagebreak

\noindent{{\textbf{Step~3.}} \emph{A Hardy--Littlewood-type argument.}}
\smallskip

\noindent{Combining \eqref{E:F'(z)-upper-bound}, \eqref{E:bdy-om-bdy-D-ineq}, and from
Lemma~\ref{L:lips-bdy-ineq}, it follows that there exists a neighbourhood $V$ of $p$ (recall that $p$ is
as introduced at the beginning of this proof) and a unitary transformation $\unitary^p$ such that
\begin{align}
  \|F'(z)v\|\leq\frac{ C_1^{s/2}}{M}\frac{\|v\|}{(\delta_D(z))^{\wt s}}
  \leq M^*\frac{\|v\|}{(\vary(\unitary^p(z)))^{\wt s}}\quad\forall z\in V\cap D,\quad\forall v\in\C^m,
  \label{E:F'(z)-up-bound-dini}
\end{align}
where $\wt s:=1-(ss_*/2)\in(0,1)$ and $M^*:= (C^{\wt s}C_1^{s/2})/M>0$. From this stage, the argument will 
resemble, in part, the proof of \cite[Theorem~1.5]{banik:lcephmlritb24}\,---\,which itself is a
Hardy--Littlewood-type argument in a non-smooth setting\,---\,so we will be brief. Let us define $\mathcal D:=\unitary^p(D)$, $\mathcal V:=\unitary^p(V)$,
$\varf=(\varf_1,
\varf_2,\dots,\varf_n):=F\circ ({\unitary^p})^{-1}\vert_{\mathcal D}$, and $\varz:=\unitary^p(z)$: a new
coordinate system in $\C^m$. 
From \eqref{E:F'(z)-up-bound-dini}, by applying the chain rule we have
\begin{equation}\label{E:f'-up-bound-w}
  \|\varf'(\varz)v\|\leq \tau(\vary(\varz))\|v\|\quad\forall \varz\in \mathcal V\cap 
  \mathcal D,\quad\forall v\in\C^m, 
\end{equation}
where $\tau(x):=M^*/x^{\wt s}$ for $x> 0$.
There exists a neighbourhood $\mathcal U\Subset \mathcal V$ of $0\in\C^m$ and a constant $\delta>0$ such that 
\begin{equation}\label{E:delta-nbhd}
  \left(\bigcup_{\xi\in \mathcal U\cap\bdy \mathcal D}\overline{\mathbb{B}^m(\xi,\delta)}\right)\cap
  \mathcal D\subset \mathcal V\cap \mathcal D.
\end{equation}
Let $0<t<t'<\delta$ and $v_0:=(0,\dots,0,i)\in\C^m$. Since $\tau$ is a non-negative Lebesgue
integrable function, by \eqref{E:f'-up-bound-w} and the Fundamental Theorem of Calculus, we can conclude that for every $\xi\in 
\mathcal U\cap\bdy \mathcal D$ and every $j:1\leq j\leq n$, the limit
\begin{equation}\label{E:funda-thm-calc}
  \varlf_{\!\!\!j}(\xi):=\varf_j(\xi+t'v_0)-\lim_{t\to 0^+}\int_t^{t'} i\frac{\bdy \varf_j}
  {\bdy\varz_m}(\xi+xv_0)\,dx
\end{equation}
exists and is independent of $t'$. 
Now define $\mscrf{{p}} =(\mscrf{{p},1}, \mscrf{{p},2} ,\dots,\mscrf{{p},n}):
(\mathcal U\cap\bdy \mathcal D)\cup
\mathcal D\lrarw\overline\Omega$ by
$$\mscrf{{p}}(z)\coloneqq \begin{cases}
  \varlf(\varz)=(\varlf_{\!\!\!1}(\varz),\dots,\varlf_{\!\!\!n}(\varz)),&\text{if }\varz\in \mathcal U\cap\bdy \mathcal D,\\
  \varf(\varz), & \text{if }\varz\in \mathcal D.\end{cases}$$
We will show that $\mscrf{{p}}$ is continuous. In particular, it suffices to show that
$\mscrf{{p}}$ is continuous
on $\mathcal U\cap\bdy \mathcal D$.}
\smallskip

Let $\varepsilon>0$ be given. Since $\tau$ is Lebesgue integrable, there exists
$\kappa\in(0,\delta)$ such that $\int_0^\kappa\tau(x)\,dx< \varepsilon/3$. Hence, from 
\eqref{E:funda-thm-calc} and \eqref{E:f'-up-bound-w}, it follows that
\begin{equation}\label{E:F-uniform-ctn}
  |\varlf_{\!\!\!j}(\xi)-\varf_j(\xi+\kappa v_0)|\leq \int_0^\kappa\tau(\vary(\xi+ x v_0))\,dx
  = \int_0^\kappa\tau(x)\,dx<\varepsilon/3\quad\forall\xi\in\mathcal U\cap\bdy \mathcal D.
\end{equation}
Note that, with our choice of $\kappa$, the above estimate is independent of $\xi \in 
{\mathcal U}\cap \bdy{\mathcal D}$. Since $\kappa\in (0, \delta)$, by \eqref{E:delta-nbhd}, the
set $S(\kappa) := \{\xi + \kappa v_0: \xi \in {\mathcal U}\cap \bdy{\mathcal D}\}$ is relatively compact
in ${\mathcal D}$. As ${\sf F}|_{S(\kappa)}$ is uniformly continuous, it is easy to see\,---\,this is the
essence of the classical ``Hardy--Littlewood trick''\,---\,that there exists an $r = r(\eps)>0$ such that,
if $\|\xi_1-\xi_2\|<r$, $\xi_1, \xi_2\in {\mathcal U}\cap \bdy{\mathcal D}$, then
$|\varlf_{\!\!\!j}(\xi_1)-\varlf_{\!\!\!j}(\xi_2)|<\eps$ for $j=1,\dots, n$. We conclude that $\varlf$ is
uniformly continuous on $\mathcal U\cap\bdy \mathcal D$.
\smallskip

Now fix $\xi=\left(\xi', \zeta+i\psi_p(\xi',\zeta)\right)\in\mathcal U\cap\bdy \mathcal D$, $\zeta\in \R$,
where $\psi_p$ is as in Lemma~\ref{L:lips-bdy-ineq}. Let
$(\varz_\nu)_{\nu\geq 1}\subset(\mathcal U\cap\overline{\mathcal D})\setminus\{\xi\}$ be an
arbitrary sequence such that $\varz_\nu\lrarw\xi$ as $\nu\to\infty$. Let us define
$${\sf Z}_\nu\coloneqq \begin{cases}
  \varz_\nu, & \text{if }\varz_\nu\in\bdy\mathcal D,\\
  \varz_\nu,&\text{if }\varz_\nu=\big(\xi', \zeta+i(x+\psi_p(\xi',\zeta))\big)\,\,\text{for some $x>0,$}\\
  \pi(\varz_\nu),  &\text{otherwise},\end{cases}$$
where $\pi(\varz)=\pi(\varz',\varz_m):=\left(\varz',\re\varz_m+i\psi_p(\varz',\re\varz_m)\right)$. 
Clearly, $\pi$ is
continuous, and hence, ${\sf Z}_\nu\lrarw\xi$ as $\nu\to\infty$. Therefore, from 
\eqref{E:F-uniform-ctn}, and from the fact that $\varlf$ is uniformly
continuous on $\mathcal U\cap\bdy \mathcal D$, we have 
\[
  \lim_{\nu\to\infty}\mscrf{{p}}(\varz_\nu) =
  \lim_{\nu\to\infty}\mscrf{{p}}({\sf Z}_\nu)=\varlf(\xi)=\mscrf{{p}}(\xi).
\]
As $\xi\in \mathcal U\cap\bdy \mathcal D$ was arbitrarily chosen (and as $(\varz_\nu)_{\nu\geq 1}$, with the
stated properties, was arbitrary) we conclude that $\mscrf{{p}}$ is continuous. Hence,
since $\unitary^p$ is an automorphism of $\C^m$, $F$ extends continuously to $(U_p\cap\bdy D)\cup 
D$, where $U_p:=(\unitary^p)^{-1}(\mathcal U)$ is a neighbourhood of $p$. Thus, in view of the
discussion at the beginning of this proof, the result follows.
\end{proof}

Next, we provide:

\begin{proof}[The proof of Theorem~\ref{T:proper-map_str-pseud}]
Since, by definition, $D$ and $\Omega$ have $\smoo^2$-smooth boundaries, $\bdy{D}$ and $\bdy\Omega$
are Lipschitz manifolds. Since $D$ and $\Omega$ are strongly pseudoconvex, they admit $\smoo^2$-smooth
defining functions that are strongly plurisubharmonic. Thus\,---\,see \cite[Section~2]{charabati:hrscmae15}\,---\,both
$D$ and $\Omega$ are \emph{regular} strongly hyperconvex Lipschitz domains. Hence, all the conditions in the
hypothesis of Theorem~\ref{T:proper-map} are satisfied, from which the result follows.  
\end{proof}

\smallskip

\section{The proofs of Theorem~\ref{T:complex-geodesic} and Corollary~\ref{C:strongly-convex-geodesic}}
\label{S:complex-geodesic_proofs}
To prove Theorem~\ref{T:complex-geodesic}, we will need the following result:
\begin{result}[Mercer, {\cite{mercer:cgigmcd93}}]\label{r:Mercer}
 Let $\Omega$ be a bounded convex domain in $\C^n$ and let $\psi:\D\lrarw\Omega$ be a
complex geodesic. There exists a constant $\beta > 1$ and constants $C_1, C_2 > 0$ such that
\[
  C_1(1-|\zeta|)\leq \delta_\Omega(\psi(\zeta))\leq C_2(1-|\zeta|)^{1/\beta}\quad \forall\zeta\in\D.
\]
\end{result}

We are now in a position to give

\begin{proof}[The proof of Theorem~\ref{T:complex-geodesic}]
Since $\Omega$ is bounded and $\C$-strictly convex, by Proposition~\ref{P:C-strict_convex}, $\Omega$ is $B$-regular.
Therefore, since $\omega$, as given in the hypothesis, is a modulus of continuity of the canonical function, by
Proposition~\ref{P:monge-ampere}, we conclude that there exists a constant $c>0$ 
such that
\[
  k_{\Omega}(z; v) \geq 
  c\frac{\|v\|}
            {\big(\omega(\delta_{\Omega}(z))\big)^{1/2}}\quad\forall z\in\Omega.   
\]
Let $\psi:\D\lrarw\Omega$ be a complex geodesic. By the previous inequality, for every $\zeta\in\D$, we have
\begin{align*}
  c\frac{\|\psi'(\zeta)\|}{\big(\omega(\delta_{\Omega}(\psi(\zeta)))\big)^{1/2}}&\leq k_{\Omega}(\psi(\zeta); \psi'(\zeta))\leq k_\D(\zeta;1)=\frac1{1-|\zeta|^2}\leq\frac{1}{1-|\zeta|}.
\end{align*}
Since $\omega$ is monotone increasing, by Result~\ref{r:Mercer} we have
\[
  \omega\big(\delta_{\Omega}(\psi(\zeta))\big)\leq \omega\left(C_2(1-|\zeta|)^{1/{\beta}}\right)\quad\forall\zeta\in\D. 
\]
Now combining the last two inequalities we get
\begin{align}\label{E:kob-up-low-geod}
   \|\psi'(\zeta)\|\leq \frac1{c}\frac{\left(\omega(C_2\left(1-|\zeta|)^{1/{\beta}}\right)\right)^{1/2}}{1-|\zeta|}
   \quad\forall\zeta\in\D.   
\end{align}
Let $s:=1/\beta$. Define a function $\tau:[0,1)\lrarw[0,\infty]$ as follows:
$$\tau(x)\coloneqq \begin{cases}
  {\sqrt{\omega\left(C_2 x^s\right)}}/{cx},&\text{if }0<x<1,\\
  \infty, & \text{if }x=0.\end{cases}$$
Write $\psi=(\psi_1,\psi_2,\dots,\psi_n)$. From \eqref{E:kob-up-low-geod}, it follows that
\begin{equation}\label{E:up-phi-hardy}
  |\psi_j'(\zeta)|\leq \tau(1-|\zeta|) \quad \forall \zeta\in \D \; \; \text{and} \; \; \forall
  j:1\leq j\leq n.
\end{equation}
Since $\sqrt{\omega}$ satisfies a Dini condition, $\int_0^\varepsilon(\sqrt{\omega(x)}/x)\,dx<\infty$ for every $\varepsilon>0$. 
From this, by using a change of variables formula for the Lebesgue integral, 
it is elementary to see that $\tau$ is of class $\leb{1}([0,1), \Leb_1)$. Therefore, by
Result~\ref{r:bharali} and~\eqref{E:up-phi-hardy}, we conclude that for every 
$j:1\leq j\leq n$, $\psi_j$ extends continuously to $\bdy\D$. Hence, the result follows.
\end{proof}

We now undertake the proof of Corollary~\ref{C:strongly-convex-geodesic}.

\begin{proof}[The proof of Corollary~\ref{C:strongly-convex-geodesic}]
Let $\Omega_1,\dots,\Omega_N$ be bounded strongly convex domains in $\C^n$ such that $\Omega:=\bigcap_{j=1}^N\Omega_j$ is non-empty.
Clearly, $\Omega$ is a bounded $\C$-strictly convex domain in $\C^n$. 
Let $\rho_j$ be a defining function of class $\smoo^2$ for $\Omega_j$ such that the real Hessian of $\rho_j$ is strictly
positive definite at each point in ${\sf dom}(\rho_j)$. Since $\Omega_j$ is strongly convex, such a defining 
function always exists; see, for instance \cite[Lemma~3.1.4]{krantz:ftscv92}. Each $\rho_j$ satisfies all the 
properties in Definition~\ref{D:SHL-domain}-$(a)$. Now define the function $\rho:\C^n\lrarw\R$ by
\[
  \rho(z):=\max_{j:1\leq j\leq N}\rho_j(z).
\]
Then, $\rho$ is a Lipschitz plurisubharmonic function. It is elementary to see that $\rho$ satisfies both conditions in 
Definition~\ref{D:SHL-domain}-$(a)$. Hence, $\Omega$ is a strongly hyperconvex Lipschitz domain. Therefore, 
by its definition and by 
Result~\ref{R:charabati}, the canonical function for $\Omega$ belongs to
$\smoo^{0,s}(\overline{\Omega})$ for some $s\in (0,1)$. 
Now define $\omega(x):=x^s$ for $x\in[0,\infty)$. Since $\sqrt{\omega}$ satisfies a Dini condition, the result follows 
from Theorem~\ref{T:complex-geodesic}-$(b)$. 
\end{proof}
\smallskip

\section*{Acknowledgements}
R. Masanta is supported by a scholarship from the Prime
Minister's Research Fellowship (PMRF) programme (fellowship no.~0201077). Both G. Bharali and R. Masanta are
supported by a DST-FIST grant (grant no.~DST FIST-2021 [TPN-700661]).

\end{document}